\documentclass[11pt,reqno]{amsart}

\usepackage[T1]{fontenc}
\usepackage{lmodern}
\usepackage{microtype}
\usepackage[a4paper,margin=30mm]{geometry}
\usepackage{amsmath,amssymb,mathtools}
\usepackage{tikz-cd}
\usetikzlibrary{arrows.meta,positioning}
\usepackage{enumitem}
\usepackage{float}
\usepackage[colorlinks=true,linkcolor=blue,citecolor=blue,urlcolor=blue]{hyperref}
\hypersetup{
  pdftitle={Counterexamples and gap phenomena for derived delooping levels},
  pdfauthor={Liang Chen},
  pdfsubject={Derived delooping levels and finitistic dimensions},
  pdfkeywords={derived delooping level, finitistic dimension, monomial algebra, one-point extension, syzygy, stable module category}
}
\usepackage{aliascnt}
\usepackage[nameinlink,noabbrev]{cleveref}

\numberwithin{equation}{section}

\newtheorem{theorem}{Theorem}[section]

\newaliascnt{proposition}{theorem}
\newtheorem{proposition}[proposition]{Proposition}
\aliascntresetthe{proposition}

\newaliascnt{lemma}{theorem}
\newtheorem{lemma}[lemma]{Lemma}
\aliascntresetthe{lemma}

\newaliascnt{corollary}{theorem}
\newtheorem{corollary}[corollary]{Corollary}
\aliascntresetthe{corollary}

\newaliascnt{question}{theorem}

\aliascntresetthe{question}

\theoremstyle{definition}
\newaliascnt{definition}{theorem}
\newtheorem{definition}[definition]{Definition}
\aliascntresetthe{definition}

\theoremstyle{remark}
\newaliascnt{remark}{theorem}
\newtheorem{remark}[remark]{Remark}
\aliascntresetthe{remark}

\crefname{theorem}{Theorem}{Theorems}
\Crefname{theorem}{Theorem}{Theorems}
\crefname{proposition}{Proposition}{Propositions}
\Crefname{proposition}{Proposition}{Propositions}
\crefname{lemma}{Lemma}{Lemmas}
\Crefname{lemma}{Lemma}{Lemmas}
\crefname{corollary}{Corollary}{Corollaries}
\Crefname{corollary}{Corollary}{Corollaries}
\crefname{question}{Question}{Questions}
\Crefname{question}{Question}{Questions}
\crefname{definition}{Definition}{Definitions}
\Crefname{definition}{Definition}{Definitions}
\crefname{remark}{Remark}{Remarks}
\Crefname{remark}{Remark}{Remarks}

\newcommand{\dell}{\operatorname{dell}}
\newcommand{\ddell}{\operatorname{ddell}}
\newcommand{\edell}{\operatorname{edell}}
\newcommand{\sddell}{\operatorname{sub-ddell}}
\newcommand{\Findim}{\operatorname{Findim}}
\newcommand{\findim}{\operatorname{findim}}
\newcommand{\pd}{\operatorname{pd}}
\newcommand{\rad}{\operatorname{rad}}
\newcommand{\soc}{\operatorname{soc}}
\newcommand{\add}{\operatorname{add}}
\newcommand{\Add}{\operatorname{Add}}
\newcommand{\op}{\mathrm{op}}
\newcommand{\st}{\mathrm{st}}
\newcommand{\Mod}{\operatorname{Mod}}
\newcommand{\modu}{\operatorname{mod}}

\newcommand{\Hom}{\operatorname{Hom}}
\newcommand{\End}{\operatorname{End}}
\newcommand{\midst}{\mathrel{\mid_{\st}}}

\title[Counterexamples for derived delooping levels]
{Counterexamples and gap phenomena for derived delooping levels}

\author{Liang Chen}
\address{School of Mathematical Sciences, Capital Normal University, Beijing 100048, China}
\email{2210501003@cnu.edu.cn}


\begin{document}

\begin{abstract}
Guo and Igusa introduced the derived and sub-derived delooping levels as refinements of the delooping level and asked how these two invariants compare, whether the derived delooping level always coincides with the opposite big finitistic dimension, and, if not, how large their difference can be.  We answer these questions.  First, we prove that every Artin algebra \(A\) satisfies
\[
   \ddell A\leq \sddell A.
\]
This gives a positive answer to the ordinary-degree comparison problem, whereas a single explicit eight-dimensional algebra \(C\) satisfies
\[
   k\text{-}\sddell C=0<1=k\text{-}\ddell C
   \qquad(k\geq2).
\]
We then compute, for every \(k\geq1\), all \(k\)-derived and \(k\)-sub-derived delooping levels of the monomial family of Barrios--Lanzilotta--Mata.  In particular, that family satisfies
\[
   k\text{-}\ddell A_{n,s}=\Findim(A_{n,s}^{\op})
   \qquad(k\geq1).
\]
In ordinary degree, this recovers an equality previously announced by Lam.  We also construct a two-vertex, eleven-dimensional monomial algebra \(\Lambda\), over an arbitrary field, such that
\[
   \Findim(\Lambda^{\op})=1<2=\ddell\Lambda=\dell\Lambda.
\]
The same strict inequality is obtained by Gao, Liu and Xu using different algebras.  A ten-dimensional non-monomial variant with a commutative corner algebra is also given.  Finally, we construct a family of finite-dimensional monomial algebras \(B_r\) satisfying
\[
   \Findim(B_r^{\op})=1,
   \qquad
   \ddell B_r=2r+1.
\]
Thus the difference \(\ddell A-\Findim(A^{\op})\) is unbounded even among monomial algebras with fixed opposite big finitistic dimension.
\end{abstract}

\maketitle

\section{Introduction}

Let \(A\) be a finite-dimensional algebra.  The delooping level \(\dell A\), introduced by G\'elinas \cite{Gelinas}, satisfies
\[
   \Findim(A^{\op})\leq \dell A.
\]
Guo and Igusa subsequently introduced the effective, sub-derived, and derived delooping levels and proved
\[
\begin{aligned}
   \Findim(A^{\op})&=\edell A
   \leq \ddell A\leq \dell A,\\
   \Findim(A^{\op})&\leq \sddell A\leq \dell A.
\end{aligned}
\]
In \cite[Question~5.1]{GuoIgusa}, they asked how \(\ddell\) and \(\sddell\) compare, whether \(\ddell A=\Findim(A^{\op})\) always holds, and, if not, whether the difference can be quantified.  In fact, immediately after Example~2.26 they explicitly conjectured that
\[
   \ddell\Lambda=\Findim(\Lambda^{\op})
   \qquad\text{for every }\Lambda;
\]
see \cite[p.~17]{GuoIgusa}.  The absence of a strict example was emphasized again in \cite[p.~1128]{GuoSymmetry}, where it was noted that no conclusive example violating this equality was then known.

The refinement is substantial.  Kershaw and Rickard constructed a finite-dimensional algebra with infinite ordinary delooping level \cite{KershawRickard}, whereas Guo and Igusa showed that the same algebra has derived delooping level $1$ \cite[Example~3.8]{GuoIgusa}.

Guo suggested that the Barrios--Lanzilotta--Mata example might indicate that the equality
\(\Findim\Lambda=\ddell(\Lambda^{\op})\) fails in general; see
\cite[Introduction]{GuoSerial}.  The indicated family does not give a counterexample: its ordinary derived delooping level equals the opposite big finitistic dimension.  This equality was previously announced by Ryan Lam in a flash talk at Anhui University in September 2025 and in his ICRA talk in June 2026; see \cite[slides~7--8]{LamAnhui} and \cite[slides~9--15]{LamICRA}.  Our treatment in \cref{sec:BLM} gives detailed parameter-dependent calculations, including the higher-degree derived and sub-derived delooping levels.

Gao, Liu and Xu \cite[Theorem~4.4 and Corollary~4.7]{GaoLiuXu} also construct one-point extensions satisfying
\(\Findim(A^{\op})=1<2=\ddell A=\dell A\).
Their work develops a general stable two-step extension criterion for derived delooping level at most one and determines the relevant invariants for a family arising from a self-injective Nakayama quotient.  The examples in \cref{sec:small} use different two-vertex algebras, of dimensions \(11\) and \(10\), and are verified by explicit syzygy and annihilator calculations.  The construction in \cref{sec:unbounded} keeps the opposite big finitistic dimension equal to \(1\) while making the derived delooping level arbitrarily large.

\paragraph*{Note on this revision.}
The constructions in \cref{sec:small,sec:unbounded} were already included in the first arXiv version of this paper.  The present revision adds discussion of Lam's earlier announcements and the related work of Gao, Liu and Xu.

The principal results of the paper are as follows.

\begin{enumerate}[label=\textup{(\Roman*)},leftmargin=2.7em]
\item If \(M\hookrightarrow N\) is an embedding of finitely generated modules over an Artin algebra, then
\[
   \ddell M\leq \dell N.
\]
Consequently, \(\ddell M\leq\sddell M\) for every finitely generated module \(M\), and hence
\[
   \ddell A\leq\sddell A.
\]
For each \(k\geq2\), however, a single eight-dimensional algebra \(C\) satisfies
\[
   k\text{-}\sddell C=0<1=k\text{-}\ddell C.
\]
Thus the proposed universal inequality \(k\text{-}\ddell A\leq k\text{-}\sddell A\) holds for \(k=1\) and fails for every \(k\geq2\), answering the universal-inequality part of \cite[Question~5.1(1)]{GuoIgusa}.  The quantitative comparison in \cref{prop:all-degree-comparison} gives a further relation in every degree.

\item For the monomial family \(A_{n,s}\) occurring in \cite[Example~4.22]{BLM}, with \(s\geq1\) and \(n\geq s+4\), we prove, for every \(k\geq1\),
\[
\begin{split}
   k\text{-}\ddell A_{n,s}
   &=\findim A_{n,s}^{\op}
     =\Findim A_{n,s}^{\op}
     =\Bigl\lfloor\frac{s}{2}\Bigr\rfloor+1,\\
   k\text{-}\sddell A_{n,s}
   &=k\text{-}\dell A_{n,s}
     =\Bigl\lceil\frac{n}{2}\Bigr\rceil.
\end{split}
\]
The ordinary-degree equality with \(\Findim(A_{n,s}^{\op})\) recovers Lam's announced result discussed above.  The displayed formulas imply, for every \(k\geq1\),
\[
   k\text{-}\ddell A_{n,s}
   -\Findim(A_{n,s}^{\op})=0,
\]
whereas
\[
   k\text{-}\sddell A_{n,s}
   -k\text{-}\ddell A_{n,s}
   =\Bigl\lceil\frac{n}{2}\Bigr\rceil
    -\Bigl\lfloor\frac{s}{2}\Bigr\rfloor-1,
\]
which is unbounded as \(n\to\infty\) with \(s\) fixed.  Thus the family suggested in \cite{GuoSerial} does not furnish a counterexample to \(\ddell A=\Findim(A^{\op})\), although it exhibits an arbitrarily large separation between the sub-derived and derived delooping levels.

\item There exists an eleven-dimensional monomial algebra \(\Lambda\) over any field such that
\[
   \Findim(\Lambda^{\op})=1<2=\ddell\Lambda=\dell\Lambda.
\]
Equivalently, \(\Gamma=\Lambda^{\op}\) satisfies
\[
   \Findim\Gamma=1<2=\ddell(\Gamma^{\op}).
\]
This provides an alternative counterexample to the equality also disproved in \cite{GaoLiuXu}.  A ten-dimensional variant with a commutative corner algebra yields the same strict inequality but is no longer monomial.  Both examples have two vertices.

\item For every \(r\geq1\), there is a finite-dimensional monomial algebra \(B_r\) with
\[
   \Findim(B_r^{\op})=1,
   \qquad
   \ddell B_r=2r+1.
\]
Hence
\[
   \sup_A\bigl(\ddell A-\Findim(A^{\op})\bigr)=\infty
\]
already in the class of finite-dimensional monomial algebras.
\end{enumerate}

The proof of (III) is elementary.  The new vertex of a one-point extension produces a simple module whose first syzygy is a carefully chosen three-dimensional module.  A two-step annihilator obstruction rules out every possible degree-one derived-delooping resolution.  The proof of (IV) is more involved: a functor to graded modules converts any hypothetical short derived-delooping resolution into a bounded filtration by short exact extensions.  An explicit composite of stable maps, invisible to every filtration factor, then gives a contradiction.

The paper is organized as follows.  In \cref{sec:prelim} we recall the definitions and several elementary stable-category facts.  In \cref{sec:comparison} we compare derived and sub-derived delooping levels.  The Barrios--Lanzilotta--Mata family is treated in \cref{sec:BLM}.  The small strict counterexample is proved in \cref{sec:small}.  The unbounded family is constructed in \cref{sec:unbounded}.

\section{Preliminaries}\label{sec:prelim}

Throughout, \(K\) denotes an arbitrary field.  Unqualified modules are finitely generated right modules unless explicitly stated otherwise.  For an algebra \(A\), write \(\Mod A\) for the category of all right \(A\)-modules, \(\modu A\) for its full subcategory of finitely generated modules, and \(\underline{\modu}\,A\) for the stable category modulo projectives.  For a module \(X\), let \(\add(X)\) and \(\Add(X)\) denote the classes of direct summands of finite and arbitrary direct sums of copies of \(X\), respectively.  For an integer \(q\geq0\), write \(X^{(q)}\) for the direct sum of \(q\) copies of \(X\).  We use
\[
\begin{aligned}
   \findim A
   &=\sup\{\pd_A X\mid X\in\modu A,\ \pd_A X<\infty\},\\
   \Findim A
   &=\sup\{\pd_A X\mid X\in\Mod A,\ \pd_A X<\infty\}
\end{aligned}
\]
for the little and big finitistic dimensions.  The syzygy functor is denoted by \(\Omega\).  If \(X\) is a retract of \(Y\) in the stable category, we write \(X\midst Y\).  We recall the following definitions from \cite[Definitions~2.5, 2.14 and 2.22]{GuoIgusa}.

\begin{definition}
Let \(k\geq0\) and \(M\in\modu A\).  The \emph{\(k\)-delooping level} of \(M\) is
\[
 k\text{-}\dell M
 =\inf\{d\geq0\mid
     \Omega^dM\midst\Omega^{d+k}N
     \text{ for some }N\in\modu A\}.
\]
For \(k=1\), we write \(\dell M\).
\end{definition}

Following \cite[p.~8]{GuoIgusa} and \cite[p.~43]{GuoSerial}, we call a module \(M\)
\emph{infinitely deloopable} if
\[
   k\text{-}\dell M=0
   \qquad\text{for every }k\geq1.
\]
Equivalently, for every \(k\geq1\), the module \(M\) is a retract in the
stable category of some \(k\)-th syzygy.

\begin{definition}
Let \(k\geq1\).  The \emph{\(k\)-derived delooping level} of \(M\) is the least integer \(m\geq0\) for which there is an exact sequence
\[
 0\longrightarrow C_n\longrightarrow\cdots\longrightarrow C_0
 \longrightarrow M\longrightarrow0,
 \qquad n\leq m,
\]
such that
\[
   (i+k)\text{-}\dell C_i\leq m-i
   \qquad(0\leq i\leq n).
\]
It is denoted by \(k\text{-}\ddell M\).  We write \(\ddell M=1\text{-}\ddell M\).
\end{definition}

\begin{definition}
Let \(k\geq1\).  The \emph{\(k\)-sub-derived delooping level} of \(M\) is the least integer \(m\geq0\) for which there is an exact sequence
\[
 0\longrightarrow M\longrightarrow D_0\longrightarrow\cdots
 \longrightarrow D_n\longrightarrow0,
 \qquad n\leq k,
\]
such that
\[
   (k-i)\text{-}\dell D_i\leq m+i
   \qquad(0\leq i\leq n).
\]
It is denoted by \(k\text{-}\sddell M\).  When \(k=1\),
\[
   \sddell M=\inf\{\dell N\mid M\hookrightarrow N\}.
\]
\end{definition}

For an Artin algebra, each of the delooping-type algebra invariants defined above is the maximum of the corresponding module invariant over the simple right modules.  In the finite-dimensional setting, we shall repeatedly use the inequalities
\begin{equation}\label{eq:basic-bounds}
   \Findim(A^{\op})\leq \ddell A\leq \dell A,
   \qquad
   \Findim(A^{\op})\leq \sddell A\leq \dell A,
\end{equation}
proved in \cite{GuoIgusa}.  For every \(M\in\modu A\) and every \(k\geq1\), we shall also use
\begin{equation}\label{eq:degree-monotonicity}
   \ddell M\leq k\text{-}\ddell M\leq k\text{-}\dell M,
   \qquad
   \sddell M\leq k\text{-}\sddell M\leq k\text{-}\dell M.
\end{equation}
The first inequalities express monotonicity in the degree, whereas the latter inequalities follow by taking witnesses of length zero; see \cite[Remarks~2.15, 2.17, 2.23 and~2.24]{GuoIgusa}.  Taking maxima over the simple modules gives the corresponding inequalities for algebra invariants.

We shall use the following standard fact about stable retracts.

\begin{lemma}[{\cite[Lemma~0.1(1)]{Gelinas}}]\label{lem:stable-actual}
If \(X\midst Y\), then \(X\) is an actual direct summand of \(Y\oplus P\) for some finitely generated projective module \(P\).
\end{lemma}
\begin{proof}
For completeness, we recall the argument.  Represent the stable retraction by maps
\(u:X\to Y\) and \(v:Y\to X\).  Then \(1_X-vu\) factors through a finitely generated projective module, say
\[
 X\xrightarrow{a}P\xrightarrow{b}X.
\]
Thus
\[
 (v,b)\binom{u}{a}=vu+ba=1_X,
\]
so \(\binom{u}{a}:X\to Y\oplus P\) is a split monomorphism.
\end{proof}

The following characterization is immediate from G\'elinas's definition of the delooping level and
\cite[Lemma~0.1(1)]{Gelinas}; it is also recorded explicitly in
\cite[p.~207, immediately after Definition~3.1]{BLM}.

\begin{lemma}\label{lem:dell-zero}
For a finitely generated module \(X\), the following are equivalent:
\begin{enumerate}[label=\textup{(\roman*)}]
\item \(\dell X=0\);
\item \(X\) embeds into a finitely generated projective module;
\item \(X\) is torsionless.
\end{enumerate}
\end{lemma}
\begin{proof}
The equivalence of \textup{(ii)} and \textup{(iii)} is the standard characterization of finitely generated torsionless modules over an Artin algebra.  If \(\dell X=0\), then \(X\midst\Omega N\) for some finitely generated module \(N\).  By \cref{lem:stable-actual}, the module \(X\) is an actual direct summand of \(\Omega N\oplus P\) for some finitely generated projective \(P\), and hence embeds into a finitely generated projective module.

Conversely, suppose that
\[
 0\longrightarrow X\longrightarrow P\longrightarrow C\longrightarrow0
\]
is exact with \(P\) finitely generated projective.  Comparing this sequence with a projective-cover sequence
\[
 0\longrightarrow\Omega C\longrightarrow P(C)\longrightarrow C\longrightarrow0
\]
and applying Schanuel's lemma gives
\[
 X\oplus P(C)\cong\Omega C\oplus P.
\]
Thus \(X\cong\Omega C\) in the stable category, so \(X\midst\Omega C\) and \(\dell X=0\).
\end{proof}

\begin{lemma}\label{lem:syzygy-shift}
If \(\ell\text{-}\dell X\leq a\) and \(s\geq0\), then
\[
   (\ell+s)\text{-}\dell(\Omega^sX)\leq\max\{a-s,0\}.
\]
\end{lemma}
\begin{proof}
Choose \(d\leq a\) and \(N\in\modu A\) such that
\[
   \Omega^dX\midst\Omega^{d+\ell}N.
\]
If \(d\geq s\), then
\[
   \Omega^{d-s}(\Omega^sX)
   =\Omega^dX
   \midst
   \Omega^{(d-s)+(\ell+s)}N,
\]
and hence
\[
   (\ell+s)\text{-}\dell(\Omega^sX)
   \leq d-s\leq a-s.
\]
If \(d<s\), applying \(\Omega^{s-d}\) to the original stable retraction gives
\[
   \Omega^sX\midst\Omega^{s+\ell}N,
\]
so
\[
   (\ell+s)\text{-}\dell(\Omega^sX)=0.
\]
\end{proof}

\section{Derived versus sub-derived delooping levels}\label{sec:comparison}

We begin by answering the comparison problem in \cite[Question~5.1(1)]{GuoIgusa}.  In fact, the following injection estimate is stronger than the requested algebra-level inequality.

\begin{theorem}[Ordinary comparison]\label{thm:ordinary-comparison}
Let \(A\) be an Artin algebra and let \(M,N\in\modu A\).  If \(M\hookrightarrow N\), then
\[
   \ddell M\leq \dell N.
\]
Consequently, for every \(M\in\modu A\),
\[
   \ddell M\leq \sddell M,
\]
and hence
\[
   \ddell A\leq \sddell A.
\]
Consequently, every Artin algebra \(A\) satisfies
\[
   \Findim(A^{\op})\leq \ddell A\leq \sddell A\leq \dell A.
\]
\end{theorem}
\begin{proof}
The assertion is immediate if \(\dell N=\infty\).  Put \(d=\dell N<\infty\).

If \(d=0\), then \cref{lem:dell-zero} shows that \(N\), and hence \(M\), embeds into a finitely generated projective module.  Thus \(\dell M=0\), and the length-zero sequence witnesses \(\ddell M=0\).

Assume \(d\geq1\).  Let \(P\twoheadrightarrow N\) be a projective cover and let \(E\) be the inverse image of \(M\) in \(P\).  Then there is an exact sequence
\[
   0\longrightarrow\Omega N\longrightarrow E\longrightarrow M\longrightarrow0.
\]
Since \(E\subseteq P\), \cref{lem:dell-zero} gives \(\dell E=0\).  By the definition of \(d=\dell N\), there is a finitely generated module \(L\) such that
\[
   \Omega^dN\midst\Omega^{d+1}L.
\]
Equivalently,
\[
   \Omega^{d-1}(\Omega N)\midst\Omega^{(d-1)+2}L,
\]
and hence
\[
   2\text{-}\dell(\Omega N)\leq d-1.
\]
In the definition of \(\ddell M\), take
\[
   C_0=E,\qquad C_1=\Omega N,\qquad n=1,\qquad m=d.
\]
Since \(1=n\leq d=m\) and
\[
   \dell C_0=0\leq d,
   \qquad
   2\text{-}\dell C_1\leq d-1,
\]
the displayed short exact sequence is a derived-delooping witness of value \(d\) for \(M\).  Therefore
\[
   \ddell M\leq d=\dell N.
\]
Taking the infimum over all embeddings \(M\hookrightarrow N\) gives
\[
   \ddell M\leq
   \inf_{M\hookrightarrow N}\dell N
   =\sddell M.
\]
Taking maxima over the simple right \(A\)-modules yields \(\ddell A\leq\sddell A\).  Finally, the inequality
\[
   \Findim(A^{\op})\leq\ddell A
\]
holds for Artin algebras by \cite[Theorem~2.2]{GuoSymmetry}, while \(\sddell A\leq\dell A\) follows directly from the definition.  Together with the inequality proved above, this gives
\[
   \Findim(A^{\op})\leq\ddell A\leq\sddell A\leq\dell A.
\]
\end{proof}

The ordinary comparison does not extend degree by degree.  We next give a small counterexample.

\begin{theorem}\label{thm:eight-dimensional}
	Let \(Q\) be the quiver
\[
\begin{tikzcd}[
	ampersand replacement=\&,
	row sep=large,
	column sep=large
	]
	1 \arrow[r, bend left=18, "a"]
	\& 2
	\arrow[l, bend left=18, "b"]
	\arrow[loop right, min distance=12mm, "g"]
\end{tikzcd}
\]
	and let paths be multiplied from left to right. Set
	\[
	C=KQ/(aba,ag,gb,g^2).
	\]
	Then \(C\) is a monomial algebra of dimension \(8\), and, for every \(k\geq2\),
	\[
	k\text{-}\sddell C=0<1=k\text{-}\ddell C.
	\]
	In ordinary degree,
	\[
	\ddell C=\sddell C=\dell C=\Findim(C^{\op})=0.
	\]
\end{theorem}
\begin{proof}
The surviving paths are
\[
   e_1,e_2,a,b,g,ab,ba,bab.
\]
Write \(P_i=e_iC\), let \(S_i\) be the corresponding simple, and put \(T=aC\).  Then
\[
\begin{split}
   P_1&=\langle e_1,a,ab\rangle_K,\\
   P_2&=\langle e_2,b,ba,bab,g\rangle_K,\\
   T&=\langle a,ab\rangle_K.
\end{split}
\]
Direct calculation gives
\begin{equation}\label{eq:eight-syzygies}
   \Omega S_1=T,
   \qquad
   \Omega S_2\simeq P_1\oplus S_2,
   \qquad
   \Omega T\simeq T\oplus S_2.
\end{equation}
Hence \(T\) and \(S_2\) occur as direct summands of syzygies of every positive order.

We claim that \(S_1\) is not a stable summand of a second syzygy.  Let \(U\subseteq\rad P\) be a minimal first syzygy, let \(f:P'\twoheadrightarrow U\) be its projective cover, and put \(L=\ker f\subseteq\rad P'\).  Coordinatewise, the kernel of right multiplication by \(a\) on \((\rad P')e_1\) is spanned by paths of types \(ab\) and \(bab\).  Hence, if \(x\in Le_1\) satisfies \(xa=0\), then \(x=yb\) for some linear combination \(y\) of coordinate paths of types \(a\) and \(ba\).  Since \(f(P')\subseteq U\subseteq\rad P\), the images of the projective generators lie in \(\rad P\).  On \((\rad P)e_1\), right multiplication by \(a\) kills every radical basis path except those of type \(b\), which it sends to paths of type \(ba\); on \((\rad P)e_2\), right multiplication by \(ba\) is zero.  It follows that \(f(y)\) lies in the span of paths of type \(ba\), on which right multiplication by \(b\) is injective.  Since \(f(y)b=f(x)=0\), one has \(f(y)=0\), whence \(y\in L\) and \(x\in L\rad C\).  Thus the \(S_1\)-part of \(\soc L\) lies in \(L\rad C\), which excludes an actual \(S_1\)-summand.  Indeed, if \(L\simeq S_1\oplus L'\), then
\[
   L\rad C=0\oplus L'\rad C,
\]
so a nonzero generator of the \(S_1\)-summand would lie outside \(L\rad C\).  By \cref{lem:stable-actual} and Krull--Schmidt, it also excludes a stable one.

The exact sequence
\[
   0\longrightarrow S_1\longrightarrow T\longrightarrow S_2\longrightarrow0
\]
witnesses \(k\text{-}\sddell S_1=0\) for every \(k\geq1\), and \eqref{eq:eight-syzygies} gives the same value for \(S_2\).  Hence \(k\text{-}\sddell C=0\).

For \(k\geq2\), the preceding claim excludes \(k\text{-}\dell S_1=0\), because every \(k\)-th syzygy is a second syzygy: for every module \(N\),
\[
   \Omega^kN=\Omega^2(\Omega^{k-2}N).
\]
A value-zero derived witness has length zero, so \(k\text{-}\ddell S_1=0\) would be equivalent to \(k\text{-}\dell S_1=0\).  On the other hand, \(\Omega S_1=T\) is infinitely deloopable.  More explicitly, the exact sequence
\[
   0\longrightarrow T\longrightarrow P_1\longrightarrow S_1\longrightarrow0
\]
is a \(k\)-derived-delooping witness of value \(1\), and hence \(k\text{-}\ddell S_1\leq1\).  Thus \(k\text{-}\ddell S_1=1\), while \(k\text{-}\ddell S_2=0\).

For \(k=1\), the embeddings \(S_1=(ab)C\subseteq P_1\) and \(S_2=gC\subseteq P_2\) show that both simples are torsionless.  Thus \(\dell C=0\), and the remaining equalities follow from \eqref{eq:basic-bounds}.
\end{proof}

\begin{remark}
The higher-degree counterexample in \cref{thm:eight-dimensional} does not concern the equality \(\ddell A=\Findim(A^{\op})\): in ordinary degree both sides are zero.  In particular, one must not confuse \(2\text{-}\ddell C=1\) with \(\ddell C=1\).
\end{remark}

\begin{proposition}\label{prop:all-degree-comparison}
Let \(A\) be an Artin algebra.  For every \(k\geq1\) and every \(M\in\modu A\),
\[
   k\text{-}\ddell M
   \leq k\bigl(k\text{-}\sddell M+1\bigr).
\]
Thus finite \(k\)-sub-derived delooping level implies finite \(k\)-derived delooping level.
\end{proposition}
\begin{proof}
If \(k\text{-}\sddell M=\infty\), there is nothing to prove.  Put
\[
   m=k\text{-}\sddell M<\infty
\]
and choose a witnessing sequence
\[
 0\longrightarrow M\longrightarrow D_0\longrightarrow\cdots
 \longrightarrow D_n\longrightarrow0,
 \qquad n\leq k.
\]
If \(n=0\), then \(M\simeq D_0\), and hence
\[
   k\text{-}\ddell M\leq k\text{-}\dell M\leq m\leq k(m+1).
\]
Assume henceforth that \(n\geq1\).  Put \(K_0=M\), and for \(1\leq i\leq n\) let
\[
   K_i=\operatorname{Im}(D_{i-1}\longrightarrow D_i).
\]
Thus \(K_n=D_n\), and the witnessing sequence splits into short exact sequences
\[
   0\longrightarrow K_i\longrightarrow D_i\longrightarrow K_{i+1}\longrightarrow0
   \qquad(0\leq i<n).
\]
For each \(0\leq i<n\), an iterated horseshoe argument gives projective modules \(P_i'\) and \(Q_i\) and an actual exact sequence
\[
 0\longrightarrow\Omega^{i+1}K_{i+1}\oplus P_i'
 \longrightarrow\Omega^iK_i\oplus Q_i
 \longrightarrow\Omega^iD_i\longrightarrow0.
\]
For every \(0\leq i\leq n\) with \(i<k\), the witnessing condition gives
\[
   (k-i)\text{-}\dell D_i\leq m+i.
\]
Applying \cref{lem:syzygy-shift} with \(\ell=k-i\), \(s=i\), and \(a=m+i\), we obtain
\[
   k\text{-}\dell(\Omega^iD_i)\leq m.
\]
Consequently, by \eqref{eq:degree-monotonicity},
\[
   k\text{-}\ddell(\Omega^iD_i)
   \leq k\text{-}\dell(\Omega^iD_i)
   \leq m
   \qquad(i<k).
\]
Adjoining projective direct summands does not change the \(k\)-derived delooping level.  Therefore, iterating the extension estimate in \cite[Lemma~3.1]{GuoIgusa} along these actual short exact sequences gives
\[
   k\text{-}\ddell M
   \leq \sum_{i=0}^{n}k\text{-}\ddell(\Omega^iD_i)+n.
\]
If \(n<k\), the right-hand side is at most \((n+1)m+n\leq k(m+1)\).  If \(n=k\), the last term \(\Omega^kD_k\) is itself a \(k\)-th syzygy, so its \(k\)-derived delooping level is zero; hence the right-hand side is at most \(km+k=k(m+1)\).
\end{proof}

\section{The Barrios--Lanzilotta--Mata family}\label{sec:BLM}

We revisit the monomial family introduced in \cite[Example~4.22]{BLM}.  The purpose of that example was to produce an arbitrarily large gap between the ordinary delooping level and the opposite finitistic dimension, and its basic syzygy pattern was recorded there.  Lam previously announced that the ordinary derived delooping level of this family equals its opposite big finitistic dimension; see \cite[slides~7--8]{LamAnhui} and \cite[slides~9--15]{LamICRA}.  The calculations below recover that equality, give explicit parameter-dependent values, and determine all \(k\)-derived and \(k\)-sub-derived delooping levels.

For general background on the path-generated structure of syzygies over monomial relation algebras, see \cite[Theorem~I]{HZ}, applied to the opposite algebra; the calculations of the relevant uniserial modules needed here are given explicitly below.  The resulting formulas confirm the qualitative gap phenomenon in \cite[Example~4.22]{BLM} and show that \(k\text{-}\sddell A_{n,s}-k\text{-}\ddell A_{n,s}\) can be arbitrarily large while both terms remain finite.  In particular, the ordinary-degree equality announced by Lam shows that this family does not produce the strict inequality suggested in \cite[Introduction]{GuoSerial}.

Let \(s\geq1\) and \(n\geq s+4\).  Let \(Q_{n,s}\) be the quiver
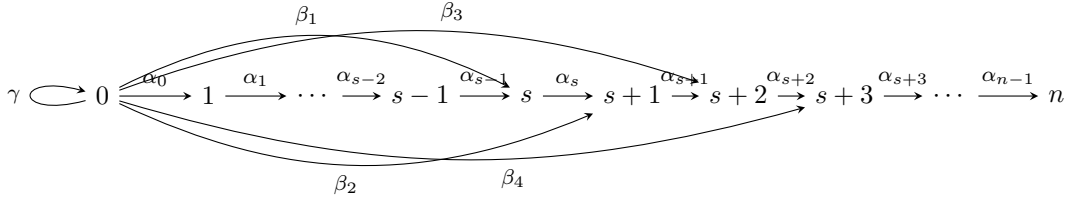
\begin{center}
\begin{tikzpicture}[
   >=stealth,
   x=1cm,
   y=1cm,
   every node/.style={font=\small},
   every edge/.style={draw,->}
]
  \node (v0)  at (0,0)    {$0$};
  \node (v1)  at (1.40,0) {$1$};
  \node (d1)  at (2.80,0) {$\cdots$};
  \node (vsm) at (4.20,0) {$s-1$};
  \node (vs)  at (5.60,0) {$s$};
  \node (vsp) at (7.00,0) {$s+1$};
  \node (vsp2) at (8.40,0) {$s+2$};
  \node (vsp3) at (9.80,0) {$s+3$};
  \node (d2)  at (11.20,0) {$\cdots$};
  \node (vn)  at (12.60,0) {$n$};

  \draw[->] (v0)  -- node[above,font=\scriptsize] {$\alpha_0$} (v1);
  \draw[->] (v1)  -- node[above,font=\scriptsize] {$\alpha_1$} (d1);
  \draw[->] (d1)  -- node[above,font=\scriptsize] {$\alpha_{s-2}$} (vsm);
  \draw[->] (vsm) -- node[above,font=\scriptsize] {$\alpha_{s-1}$} (vs);
  \draw[->] (vs)  -- node[above,font=\scriptsize] {$\alpha_s$} (vsp);
  \draw[->] (vsp) -- node[above,font=\scriptsize] {$\alpha_{s+1}$} (vsp2);
  \draw[->] (vsp2) -- node[above,font=\scriptsize] {$\alpha_{s+2}$} (vsp3);
  \draw[->] (vsp3) -- node[above,font=\scriptsize] {$\alpha_{s+3}$} (d2);
  \draw[->] (d2) -- node[above,font=\scriptsize] {$\alpha_{n-1}$} (vn);

  \draw[->] (v0) to[loop left, min distance=10mm, looseness=8]
       node[left,font=\scriptsize] {$\gamma$} (v0);
  \draw[->] (v0) to[bend left=27]
       node[pos=.48,above,font=\scriptsize] {$\beta_1$} (vs);
  \draw[->] (v0) to[bend right=24]
       node[pos=.48,below,font=\scriptsize] {$\beta_2$} (vsp);
  \draw[->] (v0) to[bend left=19]
       node[pos=.58,above,font=\scriptsize] {$\beta_3$} (vsp2);
  \draw[->] (v0) to[bend right=16]
       node[pos=.58,below,font=\scriptsize] {$\beta_4$} (vsp3);
\end{tikzpicture}
\end{center}
Formally, the vertex set is \(\{0,1,\ldots,n\}\), and the arrows are
\[
   \alpha_i:i\longrightarrow i+1\quad(0\leq i<n),
   \qquad
   \gamma:0\longrightarrow0,
   \qquad
   \beta_j:0\longrightarrow s+j-1\quad(1\leq j\leq4).
\]
Thus, when \(s=1\), the arrows \(\alpha_0\) and \(\beta_1\) are distinct parallel arrows from \(0\) to \(1\).  The displayed diagram is interpreted schematically at the boundary values of the parameters: coincident displayed vertices are identified and empty dotted portions are omitted.  Let paths be multiplied from left to right.  Define
\[
   A_{n,s}=KQ_{n,s}/I_{n,s},
\]
where \(I_{n,s}\) is generated by
\begin{equation}\label{eq:BLM-relations}
\begin{gathered}
   \alpha_i\alpha_{i+1}\alpha_{i+2}\alpha_{i+3}
   \quad(1\leq i\leq n-4),
   \qquad
   \alpha_0\alpha_1\alpha_2,
   \qquad
   \gamma^2,\\
   \gamma\beta_j,
   \qquad
   \beta_j\alpha_{s+j-1}
   \quad(1\leq j\leq4).
\end{gathered}
\end{equation}
This is the presentation of \cite[Example~4.22]{BLM}, with the parameter denoted by \(k\) there renamed \(s\) here, since \(k\) is reserved for the delooping degree.  Put \(A=A_{n,s}\), \(P_i=e_iA\), and denote the corresponding simple right modules by \(S_i\).  In every vertical display below, composition factors are listed from top to socle.  The indecomposable projective right \(A\)-modules are the following:
\[
P_0=
\begin{matrix}
	& & \multicolumn{2}{c}{\strut 0} & & \\
	s+2 & s & 0 & 1 & s+1 & s+3 \\
	& & 1 & 2 & & \\
	& & 2 & & &
\end{matrix},
\qquad
P_i=
\begin{matrix}
 i\\[-0.2ex]
 i+1\\[-0.2ex]
 i+2\\[-0.2ex]
 i+3
\end{matrix}
\quad(1\leq i\leq n-3),
\]
\[
P_{n-2}=
\begin{matrix}
 n-2\\[-0.2ex]
 n-1\\[-0.2ex]
 n
\end{matrix},
\qquad
P_{n-1}=
\begin{matrix}
 n-1\\[-0.2ex]
 n
\end{matrix},
\qquad
P_n=n.
\]
If some labels in the diagram of \(P_0\) coincide (as can happen for \(s=1,2\)), the corresponding positions represent distinct composition factors.  In particular,
\[
   \dim_K P_0=10,
   \qquad
   \dim_K A=10+4(n-3)+3+2+1=4n+4.
\]
The uniserial modules used repeatedly below are
\[
   T=\gamma A=
   \begin{matrix}
      0\\[-0.2ex]
      1\\[-0.2ex]
      2
   \end{matrix},
   \qquad
   U_j=
   \begin{matrix}
      j\\[-0.2ex]
      j+1
   \end{matrix}
   \quad(1\leq j<n).
\]
Put
\[
   r=\Bigl\lfloor\frac{s}{2}\Bigr\rfloor,
   \qquad m=r+1,
   \qquad d=\Bigl\lceil\frac{n}{2}\Bigr\rceil.
\]

\begin{theorem}\label{thm:BLM-exact}
For every \(k\geq1\),
\[
\begin{aligned}
   k\text{-}\ddell A
   &=\findim A^{\op}
   =\Findim A^{\op}
   =\Bigl\lfloor\frac{s}{2}\Bigr\rfloor+1,\\
   k\text{-}\sddell A
   &=k\text{-}\dell A
   =\Bigl\lceil\frac{n}{2}\Bigr\rceil.
\end{aligned}
\]
\end{theorem}

\subsection{Infinitely deloopable terms and the derived upper bound}

Translating the basic syzygy calculations in \cite[Example~4.22]{BLM} into the present notation gives
\begin{equation}\label{eq:BLM-basic-syzygies}
 \Omega S_0=T\oplus U_1\oplus\bigoplus_{j=s}^{s+3}S_j,
 \qquad
 \Omega T=T\oplus\bigoplus_{j=s}^{s+3}S_j.
\end{equation}
Thus \(T\) and \(S_s,S_{s+1},S_{s+2},S_{s+3}\) are direct summands of syzygies of every positive order.  For \(i+4\leq n\),
\begin{equation}\label{eq:BLM-shift}
   \Omega^2S_i=S_{i+4}\qquad(i\geq1).
\end{equation}
It follows that every \(S_i\) with \(i\geq s\) is infinitely deloopable.  If \(1\leq i<s\), choose \(a=\lceil(s-i)/4\rceil\).  Then \(i+4a\in\{s,s+1,s+2,s+3\}\), and hence
\begin{equation}\label{eq:BLM-positive-bound}
   k\text{-}\dell S_i
   \leq2\Bigl\lceil\frac{s-i}{4}\Bigr\rceil
   \leq\Bigl\lfloor\frac{s}{2}\Bigr\rfloor+1=m.
\end{equation}

Assume first that \(s\geq2\), so \(r\geq1\).  Put
\[
   W=
   \begin{matrix}
      2r-1\\[-0.2ex]
      2r\\[-0.2ex]
      2r+1
   \end{matrix}.
\]
If \(r\geq2\), there is an exact sequence
\begin{equation}\label{eq:BLM-derived-resolution}
0\longrightarrow S_{2r+1}\longrightarrow W
\longrightarrow P_{2r-3}\longrightarrow\cdots\longrightarrow P_1
\longrightarrow T\longrightarrow S_0\longrightarrow0,
\end{equation}
where the projective indices decrease by two.  For \(r=1\), the corresponding sequence is
\[
0\longrightarrow S_3\longrightarrow W
\longrightarrow T\longrightarrow S_0\longrightarrow0,
\qquad
W=
\begin{matrix}
   1\\[-0.2ex]
   2\\[-0.2ex]
   3
\end{matrix}.
\]
More explicitly, the terms in positions \(0,\ldots,r+1\) are
\[
 C_0=T,
 \qquad
 C_j=P_{2j-1}\quad(1\leq j<r),
 \qquad
 C_r=W,
 \qquad
 C_{r+1}=S_{2r+1}.
\]
The map \(T\to S_0\) is the top quotient.  Each subsequent map sends a top generator onto a generator of the preceding length-two kernel; the kernel of \(W\to U_{2r-1}\) is \(S_{2r+1}\).  Hence the displayed sequences are exact.

Now \(\Omega W=S_{2r+2}\), while \(2r+1\geq s\).  Therefore both the leftmost simple and \(\Omega W\) are infinitely deloopable.  The displayed sequence gives
\[
   k\text{-}\ddell S_0\leq m.
\]
When \(s=1\), let \(V_0\) be the length-two module along the arrow \(\alpha_0\):
\[
   V_0=
   \begin{matrix}
      0\\[-0.2ex]
      1
   \end{matrix}.
\]
Use the exact sequence
\[
   0\longrightarrow S_1\longrightarrow V_0
   \longrightarrow S_0\longrightarrow0.
\]
A direct calculation gives
\[
   \Omega V_0\simeq
   T\oplus S_1\oplus S_2^{(2)}\oplus S_3\oplus S_4.
\]
Hence \(k\text{-}\dell V_0\leq1\), while \((k+1)\text{-}\dell S_1=0\).  Consequently
\begin{equation}\label{eq:BLM-ddell-upper}
   k\text{-}\ddell A\leq m.
\end{equation}

\subsection{The matching opposite projective dimension}

We now construct a finite-dimensional left \(A\)-module of projective dimension \(m\).  All maps below are given by right multiplication by the indicated paths and are homomorphisms of left modules.

For \(s\geq2\), let
\[
   p_r=\alpha_{2r-1}\alpha_{2r}\alpha_{2r+1},
   \qquad
   L_r=Ap_r.
\]
Since \(2r\leq s\), every arrow \(\beta_j\) ends at a vertex at least \(2r\); in particular, no \(\beta_j\) ends at a vertex \(\leq2r-1\).  Thus the kernels in the recursive resolutions below come entirely from the linear \(\alpha\)-chain.
For \(r=1\), the projective cover \(Ae_1\to L_1\) has kernel \(A\alpha_0\simeq Ae_0\), so \(\pd L_1=1\).  For \(r=2\), the cover \(Ae_3\to L_2\) has kernel \(A\alpha_2\), whose cover \(Ae_2\to A\alpha_2\) has kernel \(A(\alpha_0\alpha_1)\simeq Ae_0\); hence \(\pd L_2=2\).  For \(r\geq3\), there are exact sequences
\[
 0\longrightarrow A\alpha_{2r-2}\longrightarrow Ae_{2r-1}
 \longrightarrow L_r\longrightarrow0,
\]
\[
 0\longrightarrow L_{r-2}\longrightarrow Ae_{2r-2}
 \longrightarrow A\alpha_{2r-2}\longrightarrow0.
\]
All kernels are nonzero radical submodules, so the steps are minimal.  Induction gives \(\pd L_r=r\).  The quotient
\[
   N_r=Ae_{2r+2}/L_r
\]
has projective dimension \(r+1=m\).  For \(s=1\), the quotient \(Ae_1/A\alpha_0\) has projective dimension one.  Combining these lower bounds with \eqref{eq:basic-bounds}, \eqref{eq:degree-monotonicity}, and \eqref{eq:BLM-ddell-upper} gives
\begin{equation}\label{eq:BLM-squeeze}
 m\leq\findim A^{\op}\leq\Findim A^{\op}
 \leq\ddell A\leq k\text{-}\ddell A\leq m.
\end{equation}
This proves the first assertion of \cref{thm:BLM-exact}.

\subsection{Every overmodule of \texorpdfstring{$S_0$}{S0}}

The sub-derived lower bound requires control of every embedding \(S_0\hookrightarrow M\).

\begin{lemma}\label{lem:BLM-syzygy-types}
Every first syzygy is a module over the linear Nakayama quotient
\[
   A/(\gamma,\beta_1,\ldots,\beta_4).
\]
If a nonprojective length-two uniserial module \(U_j\) occurs as a direct summand of a \(t\)-th syzygy, \(t\geq1\), then \(j\geq2t-1\).
\end{lemma}
\begin{proof}
The radical of every indecomposable projective is killed by \(\gamma\) and the \(\beta_j\).  Hence first syzygies decompose into uniserial modules supported on consecutive vertices of the linear quotient.  To justify the required summand reduction, let \(X\) be uniserial and let
\[
   f=(f_1,\ldots,f_q):X\longrightarrow\bigoplus_{j=1}^qY_j
\]
be an embedding into a finite direct sum of uniserial modules.  The submodules \(\ker f_j\) of \(X\) are linearly ordered.  Since their intersection is \(\ker f=0\), the smallest one is zero; hence one coordinate map \(f_j\) is injective.  Thus each indecomposable first-syzygy summand embeds in a uniserial direct summand of the radical of a projective and is therefore a suffix of that uniserial summand.

The nonprojective suffixes that occur are \(T\), simple modules, and uniserial modules of lengths two or three.  Among these types, only a length-two module can have a nonprojective length-two syzygy, and
\[
   \Omega U_j\cong U_{j+2}
   \qquad(1\leq j\leq n-3).
\]
Length-three uniserial modules whose tops are supported at vertices at least \(1\) have simple syzygies; simple modules supported at vertices at least \(1\) have length-three or projective syzygies; and \(\Omega T\) has only \(T\) and simple summands.  The remaining boundary terms are projective or have simple syzygies.  The assertion now follows by induction: at the first-syzygy stage every nonprojective length-two summand is some \(U_j\) with \(j\geq1\); and if \(U_j\) occurs one stage later, then its only possible nonprojective length-two predecessor is \(U_{j-2}\), so the lower bound on the index increases by two.
\end{proof}

\begin{lemma}\label{lem:BLM-forced-U1}
If \(S_0\hookrightarrow M\), then \(U_1=
\begin{matrix}
	1\\[-0.2ex]
	2
\end{matrix}\) is a direct summand of \(\Omega M\).
\end{lemma}
\begin{proof}
Let \(f:P\twoheadrightarrow M\) be a projective cover and put \(L=\ker f\subseteq\rad P\).  Choose a nonzero socle element \(x\in Me_0\) spanning the given copy of \(S_0\), and lift it to \(y\in Pe_0\).  Then
\[
   y\alpha_0\in Le_1,
   \qquad
   y\alpha_0\alpha_1\in Le_2.
\]
Multiplication by \(\alpha_0\alpha_1\) is injective on \(Pe_0\): in each \(P_0\)-coordinate it sends the independent vectors \(e_0,\gamma\) to the independent paths \(\alpha_0\alpha_1,\gamma\alpha_0\alpha_1\).  Moreover,
\[
   (\rad P)e_1\alpha_1\alpha_2=0.
\]
Decompose \(L\) into indecomposable uniserial modules supported on consecutive vertices, using \cref{lem:BLM-syzygy-types}.  If \(U_1\) were absent, every contribution to \(Le_1\alpha_1\) would come from a summand beginning at vertex \(0\); hence
\[
   Le_1\alpha_1=Le_0\alpha_0\alpha_1.
\]
Thus \(y\alpha_0\alpha_1=z\alpha_0\alpha_1\) for some \(z\in Le_0\).  Injectivity gives \(y=z\), contradicting \(f(y)=x\neq0\).
\end{proof}

\begin{lemma}\label{lem:BLM-overmodule-lower}
If \(S_0\hookrightarrow M\), then
\[
   \dell M\geq d=\Bigl\lceil\frac{n}{2}\Bigr\rceil.
\]
\end{lemma}
\begin{proof}
The socle of \(A\) has no \(S_0\)-summand, so \(M\) is not torsionless.  For \(1\leq t<d\), \cref{lem:BLM-forced-U1} implies that \(\Omega^tM\) contains the nonprojective summand \(U_{2t-1}\).  By \cref{lem:BLM-syzygy-types}, this module cannot occur in any \((t+1)\)-st syzygy, whose nonprojective length-two summands begin at vertices at least \(2t+1\).  Krull--Schmidt and \cref{lem:stable-actual} therefore exclude the stable retraction required at exponent \(t\).
\end{proof}

Taking the infimum over all modules containing \(S_0\) gives \(\sddell S_0\geq d\).  Conversely, in \eqref{eq:BLM-basic-syzygies} all summands other than \(U_1\) are infinitely deloopable, and their syzygies remain infinitely deloopable.  Using \(\Omega U_j\cong U_{j+2}\) together with the right-boundary syzygies gives
\[
   \Omega^{d-1}U_1\cong
   \begin{cases}
      U_{n-1}=P_{n-1},& n\text{ is even},\\
      S_n=P_n,& n\text{ is odd}.
   \end{cases}
\]
Hence \(\Omega^dS_0\), up to projective summands, is a direct sum of infinitely deloopable modules.  Thus
\[
   k\text{-}\dell S_0\leq d
   \qquad(k\geq1).
\]
Since \(n\geq s+4\), we also have \(m\leq d\).  Together with \eqref{eq:BLM-positive-bound} and \eqref{eq:degree-monotonicity}, this gives
\[
 d\leq\sddell A\leq k\text{-}\sddell A
 \leq k\text{-}\dell A\leq d.
\]
This completes the proof of \cref{thm:BLM-exact}.

\begin{corollary}\label{cor:BLM-gap}
For every integer \(t\geq1\), take \(s=2\) and \(n=2t+4\).  Then \(\dim_KA_{n,s}=8t+20\) and, for every \(k\geq1\),
\[
 k\text{-}\ddell A_{n,s}=\Findim A_{n,s}^{\op}=2,
 \qquad
 k\text{-}\sddell A_{n,s}=k\text{-}\dell A_{n,s}=t+2.
\]
In particular, every positive finite value of
\[
 k\text{-}\sddell A-k\text{-}\ddell A
\]
occurs while both invariants remain finite.
\end{corollary}

\begin{remark}\label{rem:Guo-suggestion}
The calculation in \cref{thm:BLM-exact} shows that the specific lead mentioned in \cite[Introduction]{GuoSerial} does not produce a counterexample to
\[
   \Findim\Lambda=\ddell(\Lambda^{\op}).
\]
Indeed, throughout the explicitly presented Barrios--Lanzilotta--Mata family one has \(\ddell A_{n,s}-\Findim(A_{n,s}^{\op})=0\).  The general suspicion expressed there is nevertheless correct: strict examples are given in \cite{GaoLiuXu} and in the next section.
\end{remark}

\begin{remark}\label{rem:BLM-published-bound}
Our formulas refer to the bound-quiver presentation in
\cite[Example~4.22, p.~217]{BLM}.  We write \(\kappa\) for the parameter denoted by \(k\) there.  On p.~217 of the journal version, the paragraph immediately preceding Example~4.22 states
\[
   \findim(A_n^{\op})
   \leq
   \left\lfloor\frac{\kappa}{4}\right\rfloor+1
\]
and the last sentence on the same page asserts, for the first case in Example~4.22, that
\[
   \pd M
   \leq
   2\left\lceil\frac{\kappa}{4}\right\rceil
\]
whenever \(\pd M<\infty\).

However, when \(s=4\) (corresponding to \(\kappa=4\) in \cite{BLM}) and \(n=10\), the left \(A\)-module
\[
   N=N_2=Ae_6/L_2
\]
has the minimal projective resolution
\[
0\longrightarrow Ae_0
\xrightarrow{\cdot\alpha_0\alpha_1}Ae_2
\xrightarrow{\cdot\alpha_2}Ae_3
\xrightarrow{\cdot\alpha_3\alpha_4\alpha_5}Ae_6
\longrightarrow N\longrightarrow0.
\]
The successive kernels are \(A\alpha_2\subseteq Ae_3\) and
\(A(\alpha_0\alpha_1)\subseteq Ae_2\), in accordance with the
path-annihilator description in \cite[Theorem~I]{HZ}.
Thus \(\pd_A N=3\).  As a right \(A^{\op}\)-module, \(N\) has top supported at vertex \(6\), so it belongs to the first case considered in \cite[Example~4.22]{BLM}, whereas both displayed upper bounds are equal to \(2\).  Thus, for the bound-quiver presentation and path-composition convention printed there, these numerical upper estimates cannot hold as stated.  Nevertheless, the exact formulas in \cref{thm:BLM-exact} confirm the qualitative conclusion of \cite[Example~4.22]{BLM}: for fixed \(s\), the ordinary gap
\[
   \dell A_{n,s}-\Findim(A_{n,s}^{\op})
\]
is unbounded as \(n\to\infty\).
\end{remark}

\section{A small monomial counterexample to the equality}\label{sec:small}

Gao, Liu and Xu \cite[Theorem~4.4 and Corollary~4.7]{GaoLiuXu} obtain counterexamples to the equality conjectured in \cite[p.~17]{GuoIgusa}.  Here we give different two-vertex counterexamples: an eleven-dimensional monomial algebra and a ten-dimensional variant with a commutative corner algebra.  We give a direct proof using explicit syzygy and annihilator calculations.

Let
\[
   R=K\langle x,y\rangle/(x,y)^3,
   \qquad
   J=\rad R,
\]
and let
\[
   M=R/(yR+xyR)
\]
be a right \(R\)-module.  It has a basis \(m_0,m_1,m_2\) satisfying
\[
   m_0x=m_1,
   \qquad
   m_1x=m_2,
   \qquad
   m_2x=0,
   \qquad
   My=0.
\]
In particular,
\begin{equation}\label{eq:M-obstruction}
   My=0,
   \qquad
   Mx^2\neq0.
\end{equation}
Define the one-point extension
\begin{equation}\label{eq:small-algebra}
   \Lambda=
   \begin{pmatrix}
      K&M\\
      0&R
   \end{pmatrix}.
\end{equation}
Equivalently, \(\Lambda\) is the bound quiver algebra with vertices \(1,2\), an arrow \(a:2\to1\), loops \(x,y\) at vertex \(1\), and monomial relations
\[
   \text{all words of length three in }x,y,
   \qquad ay,
   \qquad axy.
\]
A path basis is
\[
 e_1,e_2,x,y,x^2,xy,yx,y^2,a,ax,ax^2,
\]
so \(\dim_K\Lambda=11\).

Write \(P_i=e_i\Lambda\), and let \(S_i\) be the corresponding simple tops.  We retain vertex \(1\) for the original algebra \(R\) and use vertex \(2\) for the adjoined point.  Modules supported at vertex \(1\) will be identified with right \(R\)-modules.  Then
\begin{equation}\label{eq:small-first-syzygy}
   P_1=R,
   \qquad
   \rad P_2=M,
   \qquad
   \Omega_\Lambda S_2=M.
\end{equation}

\begin{theorem}\label{thm:small-counterexample}
For the eleven-dimensional monomial algebra \(\Lambda\) in \eqref{eq:small-algebra},
\[
   \Findim(\Lambda^{\op})=1<2=\ddell\Lambda=\dell\Lambda.
\]
\end{theorem}

\subsection{The opposite big finitistic dimension}\label{subsec:small-Findim}

The unique simple right \(R\)-module embeds in \(R\), for example as \(Kx^2\).  Hence \(\dell R=0\), and therefore \eqref{eq:basic-bounds} gives
\[
   \Findim(R^{\op})=0.
\]
We give a direct proof of the upper bound.  Recall the standard description of left \(\Lambda\)-modules as triples \((X,Y,\phi)\), where \(X\) is a \(K\)-vector space, \(Y\) is a left \(R\)-module, and \(\phi:M\otimes_RY\to X\); see \cite[Section~1]{FGR}.
Let \(L=(X,Y,\phi)\) be an arbitrary left \(\Lambda\)-module with \(\pd_\Lambda L<\infty\); no finite generation assumption is imposed.
The functor \((X,Y,\phi)\mapsto Y\) is exact and sends projective left \(\Lambda\)-modules to projective left \(R\)-modules.  Indeed, the \(Y\)-components of
\[
   \Lambda e_1=(M,R,\mathrm{id}),
   \qquad
   \Lambda e_2=(K,0)
\]
are \(R\) and \(0\), respectively, and every projective left \(\Lambda\)-module is a direct summand of a direct sum of copies of these two modules.  Thus
\[
   \pd_RY\leq\pd_\Lambda L<\infty,
\]
so \(Y\) is projective because \(\Findim(R^{\op})=0\).

Put
\[
   P_Y=(M\otimes_RY,Y,\mathrm{id})
       \simeq\Lambda e_1\otimes_RY.
\]
Since \(Y\) is a direct summand of a free left \(R\)-module, \(P_Y\) is projective over \(\Lambda\).  Also, for every \(K\)-vector space \(Z\), the module \((Z,0)\) is a direct sum of copies of \(\Lambda e_2\), and is therefore projective.  There is an exact sequence of left \(\Lambda\)-modules
\[
0\longrightarrow(M\otimes_RY,0)
\xrightarrow{\iota}P_Y\oplus(X,0)
\xrightarrow{\pi}L
\longrightarrow0.
\]
On the \(K\)-components, the maps are
\[
   \iota_X(u)=(u,-\phi(u)),
   \qquad
   \pi_X(u,x)=\phi(u)+x;
\]
on the \(R\)-components, they are the zero map and \(1_Y\), respectively.
These maps respect the structure maps of the triples, and the sequence is exact on both components.  Its kernel and middle term are projective, so \(\pd_\Lambda L\leq1\).  Since \(L\) was arbitrary, this proves
\[
   \Findim(\Lambda^{\op})\leq1.
\]
For the reverse inequality, consider the exact sequence
\begin{equation}\label{eq:small-pd-one}
0\longrightarrow(M,0)
\longrightarrow(M,R,\mathrm{id})
\longrightarrow(0,R,0)
\longrightarrow0.
\end{equation}
The first two terms are projective and the sequence is non-split, because its middle term is the indecomposable projective \(\Lambda e_1\) and both kernel and cokernel are nonzero.  Hence \((0,R,0)\) has projective dimension one, and
\begin{equation}\label{eq:small-Findim}
   \Findim(\Lambda^{\op})=1.
\end{equation}

\subsection{The upper bound}

Put
\[
   T=R/J,
   \qquad
   U=R/J^2.
\]
As right \(R\)-modules,
\begin{equation}\label{eq:small-R-syzygies}
   \Omega_RT=J=xR\oplus yR\simeq U^{2},
   \qquad
   \Omega_RU=J^2\simeq T^{4}.
\end{equation}
Moreover,
\[
   \Omega_RM=yR\oplus xyR\simeq U\oplus T.
\]
Together with \eqref{eq:small-first-syzygy}, this gives
\[
   \Omega_\Lambda^2S_2\simeq U\oplus T,
   \qquad
   \Omega_\Lambda^3S_2\simeq T^{4}\oplus U^{2}.
\]
Thus \(\Omega^2S_2\) is an actual direct summand of \(\Omega^3S_2\), and \(\dell S_2\leq2\).  The other simple embeds in \(P_1=R\), so \(\dell S_1=0\).  Therefore
\begin{equation}\label{eq:small-upper}
   \ddell\Lambda\leq\dell\Lambda\leq2.
\end{equation}

\subsection{The torsionless obstruction}

The two-step extension obstruction used below is related to the general stable criteria of Gao--Liu--Xu \cite[Theorems~2.4 and~3.3]{GaoLiuXu}.  For the present local corner algebra, we give a direct, self-contained proof using the following annihilator calculation.

The kernels of right multiplication by \(x\) and by \(y\) on \(R\) coincide:
\begin{equation}\label{eq:annihilator-equality}
   \{u\in R\mid ux=0\}=J^2=\{u\in R\mid uy=0\}.
\end{equation}
Indeed, write \(u=c+\alpha x+\beta y+w\), with \(w\in J^2\).  Then
\[
 ux=cx+\alpha x^2+\beta yx,
 \qquad
 uy=cy+\alpha xy+\beta y^2,
\]
and the displayed terms are linearly independent in each expression.  Hence each product is zero precisely when \(c=\alpha=\beta=0\).

Consequently, every torsionless right \(R\)-module \(X\) satisfies
\begin{equation}\label{eq:torsionless-implication}
   zy=0\quad\Longrightarrow\quad zx=0
   \qquad(z\in X).
\end{equation}
This immediately gives the key obstruction.

\begin{lemma}\label{lem:no-torsionless-extension}
For every integer \(b\geq0\), there is no exact sequence
\[
   0\longrightarrow X\longrightarrow M\oplus R^{(b)}
   \longrightarrow Y\longrightarrow0
\]
in which both \(X\) and \(Y\) are torsionless right \(R\)-modules.
\end{lemma}
\begin{proof}
Let \(m=m_0\) in the displayed copy of \(M\).  Since \(my=0\), the image \(\overline m\in Y\) satisfies \(\overline m y=0\).  By \eqref{eq:torsionless-implication}, \(\overline m x=0\), so \(mx\in X\).  But \((mx)y=0\); applying \eqref{eq:torsionless-implication} in \(X\) gives \(mx^2=0\), contradicting \eqref{eq:M-obstruction}.
\end{proof}

We now show that every degree-one derived-delooping witness would produce the forbidden sequence in \cref{lem:no-torsionless-extension}.

First note that
\[
   (\rad\Lambda)e_2=0.
\]
Hence every first syzygy is supported at vertex \(1\), and every second syzygy embeds in a direct sum of copies of \(P_1=R\).  Thus every second syzygy is torsionless as an \(R\)-module.

If \(Z\) is a stable summand of a second \(\Lambda\)-syzygy, then \cref{lem:stable-actual} and Krull--Schmidt give
\begin{equation}\label{eq:stable-second-form}
   Z\simeq Z_R\oplus P_2^{(t)},
\end{equation}
where \(Z_R\) is a torsionless right \(R\)-module.  Copies of \(P_1=R\) are absorbed into \(Z_R\).

Suppose, for a contradiction, that \(\ddell_\Lambda S_2\leq1\).  By increasing the witness value to \(1\), if necessary, and taking \(C_1=0\) when the original witness has length zero, we obtain an exact sequence
\begin{equation}\label{eq:degree-one-witness}
   0\longrightarrow C_1\longrightarrow C_0
   \longrightarrow S_2\longrightarrow0
\end{equation}
with
\[
   \dell C_0\leq1,
   \qquad
   2\text{-}\dell C_1=0.
\]
The first condition implies that \(X:=\Omega C_0\) is a stable summand of a second syzygy; since \(X\) is supported at vertex \(1\), it is a torsionless \(R\)-module.  The second condition and \eqref{eq:stable-second-form} give
\[
   C_1\simeq Y\oplus P_2^{(t)},
\]
where \(Y\) is torsionless over \(R\).

Take a projective cover
\[
   P_2^{(a)}\oplus P_1^{(b)}\twoheadrightarrow C_0
\]
for some integers \(a\geq1\) and \(b\geq0\).
Its composite with \(C_0\twoheadrightarrow S_2\) is surjective.  Since
\[
   \Hom_\Lambda(P_1,S_2)=0,
   \qquad
   \Hom_\Lambda(P_2,S_2)=K,
\]
a change of basis among the \(P_2\)-summands makes the composite the standard top projection from the first copy of \(P_2\).  Taking the inverse image of \(C_1\) yields
\begin{equation}\label{eq:pullback-sequence}
0\longrightarrow X\longrightarrow
M\oplus P_2^{(a-1)}\oplus P_1^{(b)}
\longrightarrow Y\oplus P_2^{(t)}
\longrightarrow0.
\end{equation}
At vertex \(2\), this sequence gives an isomorphism \(K^{(a-1)}\to K^{(t)}\), so \(t=a-1\).  Write \(D=M\oplus P_1^{(b)}\).  Since \(\Hom_\Lambda(P_2,Y)=0\), the surjection in \eqref{eq:pullback-sequence} has block form
\[
   \begin{pmatrix}f&0\\ g&h\end{pmatrix}:
   D\oplus P_2^{(t)}\longrightarrow Y\oplus P_2^{(t)}.
\]
If \(t=0\), there are no \(P_2\)-summands to remove.  If \(t>0\), then
\[
   \End_\Lambda(P_2^{(t)})\cong M_t(\End_\Lambda(P_2))\cong M_t(K),
\]
and the component of \(h\) at vertex \(2\) is precisely its defining matrix.  Since that component is invertible, \(h\) is an automorphism.  Precomposing with the automorphism
\[
   (d,p)\longmapsto(d,p-h^{-1}g(d))
\]
reduces the map to \(\mathrm{diag}(f,h)\).  The \(P_2^{(t)}\)-summands split off, leaving an exact sequence
\[
   0\longrightarrow X\longrightarrow M\oplus R^{(b)}
   \longrightarrow Y\longrightarrow0.
\]
This contradicts \cref{lem:no-torsionless-extension}.  Therefore \(\ddell_\Lambda S_2\geq2\).  Together with \eqref{eq:small-Findim} and \eqref{eq:small-upper}, this proves \cref{thm:small-counterexample}.

\begin{proposition}\label{prop:ten-dimensional}
Let
\[
   R_c=K[x,y]/(x,y)^3,
   \qquad
   M_c=R_c/(y),
   \qquad
   \Lambda_c=
   \begin{pmatrix}K&M_c\\0&R_c\end{pmatrix}.
\]
Then \(\dim_K\Lambda_c=10\) and
\[
   \Findim(\Lambda_c^{\op})=1<2=\ddell\Lambda_c=\dell\Lambda_c.
\]
\end{proposition}
\begin{proof}
Put
\[
   J_c=\rad R_c=(x,y),
   \qquad
   T_c=R_c/J_c,
   \qquad
   U_c=R_c/J_c^2.
\]
The unique simple right \(R_c\)-module \(T_c\) embeds in \(R_c\), for example as \(Kx^2\).  Hence \(\dell R_c=0\), and \eqref{eq:basic-bounds} gives \(\Findim(R_c^{\op})=0\).  The direct argument in \cref{subsec:small-Findim}, with \(R_c\) and \(M_c\) in place of \(R\) and \(M\), therefore gives \(\Findim(\Lambda_c^{\op})\leq1\).  Replacing \(R,M\) by \(R_c,M_c\) in \eqref{eq:small-pd-one} gives a non-split exact sequence with projective first two terms, so \(\Findim(\Lambda_c^{\op})=1\).

The lower-bound argument for the derived delooping level also remains valid: right multiplication by \(x\) and by \(y\) on \(R_c\) again have kernel \(J_c^2\), while \(M_cy=0\) and \(M_cx^2\neq0\).

For the upper bound on the delooping level,
\[
   \Omega_{\Lambda_c}^2S_2\simeq U_c,
   \qquad
   \Omega_{R_c}U_c\simeq T_c^{(3)},
\]
and
\begin{equation}\label{eq:commutative-kernel}
   \Omega_{R_c}^2T_c\simeq U_c\oplus T_c^{(4)}.
\end{equation}
To verify \eqref{eq:commutative-kernel}, consider the projective cover
\[
   R_c^2\longrightarrow J_c,
   \qquad
   (r,s)\longmapsto xr+ys.
\]
Its kernel is the direct sum of the cyclic submodule generated by \((-y,x)\), which is isomorphic to \(U_c\), and a four-dimensional semisimple complement.  Hence
\[
   \Omega_{R_c}^3U_c\simeq U_c^{(3)}\oplus T_c^{(12)},
\]
so \(U_c\) is a direct summand of a third syzygy and \(\dell S_2\leq2\).  The other simple \(S_1\) is torsionless, as before, and hence \(\dell\Lambda_c\leq2\).  The lower-bound argument gives \(\ddell\Lambda_c\geq2\), while \(\ddell\Lambda_c\leq\dell\Lambda_c\); therefore \(\ddell\Lambda_c=\dell\Lambda_c=2\).
\end{proof}

\section{Unboundedness of \texorpdfstring{\(\ddell A-\Findim(A^{\op})\)}{ddell A - Findim(Aop)} for monomial algebras}\label{sec:unbounded}

We now construct a monomial family for which the opposite big finitistic dimension is fixed while the derived delooping level tends to infinity.

Fix \(r\geq1\), and put
\begin{equation}\label{eq:Br-parameters}
   m=2r,
   \qquad
   h=2r+2,
   \qquad
   \ell=2h=4r+4,
   \qquad
   N=(r+1)\ell.
\end{equation}
Let \(Q_N\) be the linearly oriented quiver
\[
   1\longrightarrow2\longrightarrow\cdots\longrightarrow N,
\]
and let \(J_T\) be the arrow ideal of \(KQ_N\).  Put
\[
   T=KQ_N/J_T^\ell.
\]
Let \(S_i\) denote the simple right \(T\)-module at vertex \(i\).  For
\[
   1\leq s\leq N,
   \qquad
   1\leq b\leq\min\{\ell,N-s+1\},
\]
let \(I(s,b)\) denote the interval right \(T\)-module
\[
I(s,b)=
\begin{matrix}
   s\\
   s+1\\
   \vdots\\
   s+b-1
\end{matrix}.
\]
Put
\[
   U=I(1,h),
   \qquad
   V=\bigoplus_{i=1}^{N}S_i,
   \qquad
   D=K[\gamma]/(\gamma^2).
\]
Set
\[
   W=D\otimes_KU,
   \qquad
   E=W\oplus V,
\]
where \(D\) acts on \(V\) through \(D/(\gamma)\simeq K\), and define
\begin{equation}\label{eq:Br-matrix}
   B_r=
   \begin{pmatrix}
      D&E\\
      0&T
   \end{pmatrix}.
\end{equation}

Equivalently, \(B_r\) is the monomial algebra with vertices \(1,2,\ldots,N+1\).  We keep the vertices \(1,\ldots,N\) for \(T\) and use \(N+1\) for the new vertex.  The arrows are
\[
\begin{gathered}
   \gamma:N+1\to N+1,
   \qquad
   a:N+1\to1,
   \qquad
   \beta_i:N+1\to i\quad(1\leq i\leq N),\\
   \alpha_i:i\to i+1\quad(1\leq i<N).
\end{gathered}
\]
The quiver of \(B_r\) may be displayed schematically as in \cref{fig:Br-quiver}; in particular, there are two parallel arrows from \(N+1\) to \(1\), namely \(a\) and \(\beta_1\).

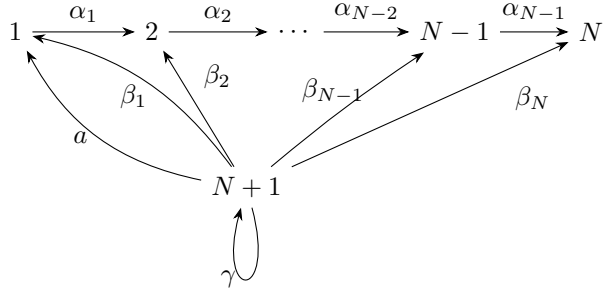
\begin{figure}[H]
\centering
\begin{tikzpicture}[>=Stealth, every node/.style={font=\small}, x=1cm, y=1cm]
   \node (1)   at (0,0) {$1$};
   \node (2)   at (1.8,0) {$2$};
   \node (dots) at (3.7,0) {$\cdots$};
   \node (Nm1) at (5.8,0) {$N-1$};
   \node (N)   at (7.6,0) {$N$};
   \node (Np1) at (3.05,-2.05) {$N+1$};

   \draw[->] (1) -- node[above] {$\alpha_1$} (2);
   \draw[->] (2) -- node[above] {$\alpha_2$} (dots);
   \draw[->] (dots) -- node[above] {$\alpha_{N-2}$} (Nm1);
   \draw[->] (Nm1) -- node[above] {$\alpha_{N-1}$} (N);

   \path[->] (Np1) edge[loop below, min distance=13mm, looseness=14] node[left] {$\gamma$} (Np1);
   \draw[->,bend left=25] (Np1) to node[left,pos=.50] {$a$} (1);
   \draw[->,bend right=20] (Np1) to node[below,pos=.56] {$\beta_1$} (1);
   \draw[->] (Np1) to node[above right,pos=.61, xshift=1pt, yshift=-1pt] {$\beta_2$} (2);
   \draw[->,bend left=3] (Np1) to node[below left,pos=.77, xshift=-4pt, yshift=4pt] {$\beta_{N-1}$} (Nm1);
   \draw[->] (Np1) to node[below right,pos=.77, xshift=0pt, yshift=-2pt] {$\beta_N$} (N);
\end{tikzpicture}
\caption{A schematic depiction of the bound quiver of \(B_r\). In addition to the displayed arrows, there are arrows \(\beta_i:N+1\to i\) for the intermediate vertices \(3\le i\le N-2\); the defining monomial relations are listed in \eqref{eq:Br-relations}.}
\label{fig:Br-quiver}
\end{figure}

The defining monomial relations are
\begin{equation}\label{eq:Br-relations}
\begin{gathered}
   \gamma^2,
   \qquad
   \gamma\beta_i\quad(1\leq i\leq N),
   \qquad
   \beta_i\alpha_i\quad(1\leq i<N),\\
   a\alpha_1\cdots\alpha_h,
   \qquad
   \alpha_i\alpha_{i+1}\cdots\alpha_{i+\ell-1}
   \quad(1\leq i\leq N-\ell).
\end{gathered}
\end{equation}
In particular, \(B_r\) is finite dimensional and monomial.

\begin{theorem}\label{thm:unbounded-gap}
For every \(r\geq1\),
\[
   \Findim(B_r^{\op})=1,
   \qquad
   \ddell B_r=2r+1.
\]
Consequently,
\[
   \ddell B_r-\Findim(B_r^{\op})=2r,
\]
which is unbounded as \(r\to\infty\) among finite-dimensional monomial algebras.
\end{theorem}

\subsection{Basic syzygies}

Write \(B=B_r\), \(P_i=e_iB\), and let \(S_{N+1}\) be the simple at the new vertex.  Put
\[
   G=\gamma B.
\]
The radical of \(P_{N+1}\) decomposes as
\[
   \rad P_{N+1}=G\oplus U\oplus V,
\]
and the projective cover \(P_{N+1}\twoheadrightarrow G\), \(e_{N+1}\mapsto\gamma\), has kernel \(G\oplus V\).  Hence
\begin{equation}\label{eq:Br-basic-syzygies}
   \Omega_BS_{N+1}=G\oplus U\oplus V,
   \qquad
   \Omega_BG=G\oplus V.
\end{equation}
Iterating the second equality gives
\begin{equation}\label{eq:Br-G-iterates}
   \Omega_B^jG
   \simeq G\oplus V\oplus\Omega_TV\oplus\cdots
   \oplus\Omega_T^{j-1}V
   \qquad(j\geq1).
\end{equation}
It follows that \(G\) and every \(S_i\), \(1\leq i\leq N\), are direct summands of syzygies of every positive order.  Their syzygies have the same property.

\subsection{The opposite big finitistic dimension}

Let \(Q_i=Be_i\) denote the indecomposable left projectives, and let \(L_{N+1}\) be the simple left module at vertex \(N+1\).  For every \(1\leq i\leq N\),
\[
   \beta_iJ(B)=0,
\]
and path decomposition by the last arrow gives a left-module direct sum
\begin{equation}\label{eq:Br-rad-left}
   \rad Q_i=K\beta_i\oplus R_i,
   \qquad
   K\beta_i\simeq L_{N+1}.
\end{equation}

\begin{proposition}\label{prop:Br-Findim}
One has
\[
   \findim(B^{\op})=\Findim(B^{\op})=1.
\]
\end{proposition}
\begin{proof}
Since \(B\) is finite dimensional, it is left perfect.  Hence every left \(B\)-module has a projective cover, and every projective left \(B\)-module is a direct sum of the indecomposable projectives \(Q_1,\ldots,Q_{N+1}\).  Suppose that a left \(B\)-module has finite projective dimension \(d\geq2\), and take the end of a minimal projective resolution
\[
   0\longrightarrow F_d\xrightarrow{d_d}F_{d-1}
   \xrightarrow{d_{d-1}}F_{d-2}.
\]
If \(F_d\) had a summand \(Q_i\) with \(1\leq i\leq N\), then the restriction of the minimal map \(d_d\) to this summand would kill \(\beta_i\), because its image lies in the radical and \(\beta_iJ(B)=0\).  This contradicts injectivity.  Thus \(F_d\in\Add(Q_{N+1})\).

If \(F_{d-1}\) had a summand \(Q_i\) with \(1\leq i\leq N\), then \(K\beta_i\subseteq\ker d_{d-1}=\operatorname{im}d_d\).  By \eqref{eq:Br-rad-left}, after isolating this copy of \(Q_i\), we may write
\[
   \rad F_{d-1}=K\beta_i\oplus R'
\]
for a submodule \(R'\).  Since the resolution is minimal, \(\ker d_{d-1}\subseteq\rad F_{d-1}\), and the inclusion \(K\beta_i\subseteq\ker d_{d-1}\) therefore gives
\[
   \ker d_{d-1}=K\beta_i\oplus(\ker d_{d-1}\cap R').
\]
Thus this copy of \(L_{N+1}\) is a direct summand of \(\ker d_{d-1}\cong F_d\), which is impossible because \(F_d\) is projective whereas \(L_{N+1}\) is nonprojective.  Therefore \(F_{d-1}\in\Add(Q_{N+1})\).

Now \(e_{N+1}J(B)e_{N+1}=K\gamma\) and \(\gamma^2=0\).  Since \(d_d\) is minimal, its image is contained in \(\gamma F_{d-1}\), and consequently \(d_d(\gamma F_d)=0\).  But \(\gamma F_d\neq0\), contradicting injectivity.  Hence every left module of finite projective dimension has projective dimension at most one.

Finally, the map
\[
   Q_{N+1}\longrightarrow Q_1,
   \qquad z\longmapsto za,
\]
is injective because \(a\) and \(\gamma a\) are linearly independent.  Its image is a nonzero proper submodule of the indecomposable projective \(Q_1\), so the monomorphism is non-split.  Its cokernel is therefore nonprojective and, from the displayed projective presentation, has projective dimension exactly one.  This proves the assertion for both little and big finitistic dimensions.
\end{proof}

\subsection{A derived-delooping resolution of length \texorpdfstring{$2r+1$}{2r+1}}

Inside \(U=I(1,h)\), let \(\soc U=S_h\), and set
\begin{equation}\label{eq:Br-C}
   C=P_{N+1}/(G\oplus V\oplus\soc U).
\end{equation}
Then
\begin{equation}\label{eq:Br-C-syzygy}
   \rad C=I(1,h-1),
   \qquad
   \Omega_BC=G\oplus V\oplus S_h.
\end{equation}
Thus \(k\text{-}\dell C\leq1\) for every \(k\geq1\).

We shall use the following elementary exact sequence of interval modules.

\begin{lemma}\label{lem:interval-four-term}
Let \(3\leq b<\ell\), put \(t=s+b-1\), and assume that
\[
   s+\ell+b-3\leq N.
\]
Then all the interval modules displayed below exist, \(P_s=I(s,\ell)\) is projective, and there is an exact sequence
\begin{equation}\label{eq:interval-four-term}
0\longrightarrow I(s+\ell,b-2)
\longrightarrow I(t,\ell-1)
\longrightarrow P_s\oplus S_t
\longrightarrow I(s,b)
\longrightarrow0.
\end{equation}
\end{lemma}
\begin{proof}
Let \(p_0,\ldots,p_{\ell-1}\) be the path basis of \(P_s\), let \(u_0,\ldots,u_{b-1}\) be the standard basis of \(I(s,b)\), and let \(v\) generate \(S_t\).  Define
\[
   \pi:P_s\oplus S_t\longrightarrow I(s,b)
\]
by
\[
   \pi(p_j)=u_j\quad(j<b),
   \qquad
   \pi(p_j)=0\quad(j\geq b),
   \qquad
   \pi(v)=-u_{b-1}.
\]
Then
\[
   \ker\pi
   =\langle p_{b-1}+v,p_b,\ldots,p_{\ell-1}\rangle_K
   \simeq I(t,\ell-b+1).
\]
The natural quotient map
\[
   I(t,\ell-1)\twoheadrightarrow I(t,\ell-b+1)
\]
has kernel \(I(s+\ell,b-2)\), which proves \eqref{eq:interval-four-term}.
\end{proof}

For \(0\leq j\leq r\), put
\[
   s_j=1+j\ell,
   \qquad
   b_j=2r+1-2j,
   \qquad
   L_j=I(s_j,b_j).
\]
Thus \(L_0=I(1,h-1)\) and \(L_r=S_{1+r\ell}\).  For \(0\leq j\leq r-1\), set
\[
   t_j=s_j+b_j-1,
   \qquad
   A_j=P_{s_j}\oplus S_{t_j},
   \qquad
   R_j=I(t_j,\ell-1)=\Omega_BS_{t_j-1}.
\]
For these parameters,
\[
   N-(s_j+\ell+b_j-3)=(r-j)(\ell-2)+1>0,
\]
so \cref{lem:interval-four-term} applies and gives
\[
   0\longrightarrow L_{j+1}\longrightarrow R_j
   \longrightarrow A_j\longrightarrow L_j\longrightarrow0.
\]
Splicing these exact sequences with
\[
   0\longrightarrow L_0\longrightarrow C\longrightarrow S_{N+1}\longrightarrow0
\]
yields
\begin{equation}\label{eq:Br-long-witness}
\begin{split}
0\longrightarrow S_{1+r\ell}
&\longrightarrow R_{r-1}\longrightarrow A_{r-1}
\longrightarrow\cdots\\
&\longrightarrow R_0\longrightarrow A_0
\longrightarrow C\longrightarrow S_{N+1}\longrightarrow0.
\end{split}
\end{equation}
All terms other than \(C\) are projective or infinitely deloopable, while \(k\text{-}\dell C\leq1\).  If we number the terms preceding \(S_{N+1}\) by \(C_0=C\), \(C_{2j+1}=A_j\), \(C_{2j+2}=R_j\), and \(C_{2r+1}=S_{1+r\ell}\), then the sequence \eqref{eq:Br-long-witness} is a valid derived-delooping witness of value \(2r+1\).  Hence
\begin{equation}\label{eq:Br-upper}
   \ddell B\leq2r+1.
\end{equation}

\subsection{A graded obstruction to shorter derived resolutions}

We prove the reverse inequality in \eqref{eq:Br-upper}.  Put \(e=e_1+\cdots+e_N\).  For a right \(B\)-module \(X\), define
\begin{equation}\label{eq:F-functor}
   F(X)=\operatorname{coker}
   \bigl(Xe_{N+1}\otimes_DW\xrightarrow{\mu_X}Xe\bigr),
\end{equation}
where \(\mu_X\) is induced by multiplication.

\begin{lemma}\label{lem:F-exact}
If \(X\) is torsionless, then \(\mu_X\) is injective.  Consequently, \(F\) preserves every short exact sequence whose three terms are torsionless.  Moreover,
\[
   F(G)=0,
   \qquad
   F(P_{N+1})=V,
   \qquad
   F(P_i)=e_iT\quad(1\leq i\leq N),
\]
and \(F(X)=X\) whenever \(Xe_{N+1}=0\).
\end{lemma}
\begin{proof}
Embed \(X\) in a projective module \(P\).  Since \(W=D\otimes_KU\) is free as a left \(D\)-module, the map
\[
   Xe_{N+1}\otimes_DW\longrightarrow Pe_{N+1}\otimes_DW
\]
is injective.  The multiplication map for a projective module is injective: on \(P_{N+1}\) its image is the direct summand \(W\subseteq W\oplus V=P_{N+1}e\), and on the other indecomposable projectives the source is zero.  Hence \(\mu_X\) is injective.  A diagram chase, or the snake lemma applied to the exact rows, proves exactness of \(F\) on torsionless short exact sequences.  The displayed values follow directly from the definitions.
\end{proof}

\begin{lemma}\label{lem:first-syzygy-decomposition}
For every finitely generated right \(B\)-module \(X\), there is a decomposition
\[
   \Omega_BX\simeq G^{(a)}\oplus X'
\]
for some \(a\geq0\) and some right \(T\)-module \(X'\).
\end{lemma}
\begin{proof}
A minimal first syzygy is a submodule of the radical of a projective module, and such a radical has the form \(G^{(b)}\oplus Y\), with \(Ye_{N+1}=0\).  The vertex-\(N+1\) part of \(G^{(b)}\) is \(K^{(b)}\).  After a change of basis by an element of \(\mathrm{GL}_b(K)\), the subspace \((\Omega X)e_{N+1}\) is spanned by the first \(a\) coordinate vectors.  The submodules generated by these vectors are copies of \(G\) and split off from the ambient direct sum.  The remaining summand has zero vertex-\(N+1\) component and is therefore a \(T\)-module.
\end{proof}

Let
\[
   H=K[t]/(t^\ell),
   \qquad \deg t=1.
\]
Let \(\operatorname{grmod}H\) denote the category of finitely generated graded right \(H\)-modules with degree-zero homomorphisms.  A right \(T\)-module is regarded as an object of \(\operatorname{grmod}H\) by placing its vertex-\(i\) component in degree \(i\).  Under this identification, we use the same notation \(I(s,b)\) for the graded interval with generator in degree \(s\) and length \(b\).  The coefficient pairing
\[
   \langle t^i,t^j\rangle=
   \begin{cases}
      1,&i+j=\ell-1,\\
      0,&i+j\neq\ell-1,
   \end{cases}
\]
is associative and nondegenerate.  Thus \(H\) is graded Frobenius; in particular, every finitely generated graded projective \(H\)-module is graded injective.  Since \(H\) is graded local, every finitely generated graded projective module is a finite direct sum of degree shifts of \(H\).  We write
\[
   \underline{\Hom}_H(X,Y)
\]
for degree-zero homomorphisms modulo those which factor through finitely generated graded projective modules.

Let \(\mathcal C\) be a full additive subcategory of \(\operatorname{grmod}H\), closed under direct summands.  For \(p\geq1\), let \(\operatorname{Filt}_p(\mathcal C)\) be the class of direct summands of graded modules \(E\) admitting a filtration
\[
   0=E_0\subseteq E_1\subseteq\cdots\subseteq E_q=E,
   \qquad q\leq p,
\]
with \(E_j/E_{j-1}\in\mathcal C\) for every \(j\).  Thus \(\operatorname{Filt}_1(\mathcal C)=\mathcal C\), and the classes \(\operatorname{Filt}_p(\mathcal C)\) are increasing with \(p\).

Fix \(m=2r\), and put
\begin{equation}\label{eq:Cm}
   \mathcal C_m=
   \add\{F(P),F(\Omega_B^{m+1}X)
   \mid P\text{ projective},\ X\in\modu B\}
   \subseteq\operatorname{grmod}H.
\end{equation}

\begin{lemma}\label{lem:ddell-to-extension}
If \(\ddell_BS_{N+1}\leq m\), then
\[
   F(\Omega_B^mS_{N+1})\in\operatorname{Filt}_{m+1}(\mathcal C_m).
\]
More precisely, it is a direct summand of a graded module admitting a filtration of length at most \(m+1\), all of whose factors lie in \(\mathcal C_m\).
\end{lemma}
\begin{proof}
Choose a derived-delooping witness
\[
0\longrightarrow C_n\longrightarrow\cdots\longrightarrow C_0
\longrightarrow S_{N+1}\longrightarrow0,
\qquad n\leq m,
\]
with
\[
   (i+1)\text{-}\dell C_i\leq m-i.
\]
Put \(K_0=S_{N+1}\), write the witness as short exact sequences
\[
   0\longrightarrow K_{i+1}\longrightarrow C_i
   \longrightarrow K_i\longrightarrow0
   \qquad(0\leq i\leq n),
\]
where \(K_{n+1}=0\); in particular, \(K_n\simeq C_n\).  Set
\[
   A_i=F(\Omega_B^{m-i}C_i),
   \qquad
   W_i=F(\Omega_B^{m-i}K_i).
\]
The delooping condition implies that \(\Omega_B^{m-i}C_i\) is a stable summand of an \((m+1)\)-st syzygy: apply enough additional syzygies to a witnessing stable retraction.  By \cref{lem:stable-actual}, it is an actual summand after adding a projective module.  Hence
\[
   A_i\in\mathcal C_m
   \qquad(0\leq i\leq n).
\]
In particular, \(W_n=A_n\in\operatorname{Filt}_1(\mathcal C_m)\).

For \(0\leq i<n\), one has \(m-i\geq1\).  An iterated horseshoe argument, equivalently repeated syzygy rotation as in \cite[Lemma~2.13]{GuoIgusa}, gives projective modules \(P_i'\) and \(Q_i\) and an actual exact sequence
\[
0\longrightarrow\Omega_B^{m-i}C_i\oplus P_i'
\longrightarrow\Omega_B^{m-i}K_i\oplus Q_i
\longrightarrow\Omega_B^{m-i-1}K_{i+1}
\longrightarrow0.
\]
The projective correction on the left is harmless: set
\[
   A_i'=F(\Omega_B^{m-i}C_i\oplus P_i')
       =A_i\oplus F(P_i')\in\mathcal C_m.
\]
All three terms of the preceding exact sequence are torsionless.  The left and middle terms are positive syzygies up to projective summands.  The right term is also a positive syzygy unless \(i=m-1\); in that exceptional case \(i=n-1\), so it is \(K_n\simeq C_n\).  Since \(n=m\) and \((m+1)\text{-}\dell C_n=0\), the module \(C_n\) is a stable summand of a positive syzygy; hence \cref{lem:stable-actual,lem:dell-zero} show that it is torsionless.  Therefore \cref{lem:F-exact} yields
\[
0\longrightarrow A_i'
\longrightarrow W_i\oplus F(Q_i)
\longrightarrow W_{i+1}
\longrightarrow0.
\]
A descending induction on \(i\), starting with \(W_n\) and using these short exact sequences, shows that
\[
   W_i\in\operatorname{Filt}_{n-i+1}(\mathcal C_m).
\]
Indeed, suppose inductively that \(W_{i+1}\in\operatorname{Filt}_p(\mathcal C_m)\).  By definition, there is a graded module \(E'\) such that \(W_{i+1}\oplus E'\) admits a \(\mathcal C_m\)-filtration of length at most \(p\).  Taking the direct sum of the preceding short exact sequence with the identity sequence on \(E'\) gives
\[
0\longrightarrow A_i'
\longrightarrow W_i\oplus F(Q_i)\oplus E'
\longrightarrow W_{i+1}\oplus E'
\longrightarrow0.
\]
Pulling back a filtration of the quotient and adjoining the initial factor \(A_i'\in\mathcal C_m\) yields a \(\mathcal C_m\)-filtration of the middle term of length at most \(p+1\).  Since \(W_i\) is a direct summand of that middle term, \(W_i\in\operatorname{Filt}_{p+1}(\mathcal C_m)\).  Consequently,
\[
   F(\Omega_B^mS_{N+1})=W_0
   \in\operatorname{Filt}_{n+1}(\mathcal C_m)
   \subseteq\operatorname{Filt}_{m+1}(\mathcal C_m),
\]
as required.
\end{proof}

We now describe the objects in \(\mathcal C_m\).  For
\[
   c_s=\min\{\ell,N-s+1\},
\]
the syzygy of an interval is
\begin{equation}\label{eq:T-interval-syzygy}
   \Omega_TI(s,b)=
   \begin{cases}
      0,&b=c_s,\\
      I(s+b,c_s-b),&b<c_s.
   \end{cases}
\end{equation}
Since \(T\) is a truncated line Nakayama algebra, every indecomposable finitely generated right \(T\)-module is an interval \(I(s,b)\).  If the right boundary is not reached during the next two syzygies, \eqref{eq:T-interval-syzygy} gives
\[
   \Omega_T^2I(s,b)\cong I(s+\ell,b).
\]
Thus two syzygies preserve the length and shift the starting degree by \(\ell\).

By \cref{lem:first-syzygy-decomposition} and \eqref{eq:Br-G-iterates},
\begin{equation}\label{eq:F-high-syzygy}
 F(\Omega_B^{m+1}X)
 \simeq
 \left(\bigoplus_{j=0}^{m-1}\Omega_T^jV\right)^{\oplus a}
 \oplus\Omega_T^mX'.
\end{equation}
Decompose \(X'\) into intervals.  Since \(m=2r\), every surviving interval for which the right boundary is not reached starts, after \(m\) syzygies, in degree at least \(r\ell+1\).  If the boundary is reached, every nonzero nonprojective boundary interval starts in degree at least
\[
   N-\ell+2=r\ell+2.
\]
For a simple summand of \(V\), the successive nonprojective syzygies, before the boundary is reached, alternate between intervals of lengths \(\ell-1\) and \(1\).  Any boundary interval of an intermediate length again starts in degree at least \(r\ell+2\).  Together with the values of \(F\) on projectives, these observations prove the following.

\begin{lemma}\label{lem:Cm-types}
Every object of \(\mathcal C_m\) is a direct summand of a finite direct sum of graded intervals of the following four types:
\begin{enumerate}[label=\textup{(\roman*)}]
\item intervals of length \(1\);
\item intervals of length \(\ell-1\);
\item graded free \(H\)-modules;
\item intervals whose generating degree is at least \(r\ell+1\).
\end{enumerate}
\end{lemma}

On the other hand,
\begin{equation}\label{eq:target-interval}
   Z_0:=\Omega_T^{m-1}U=I(s,h),
   \qquad
   s=1+(r-1)\ell+h.
\end{equation}
Its last nonzero degree is
\[
   s+h-1=r\ell.
\]
By the additivity of syzygies and \eqref{eq:Br-basic-syzygies},
\[
   \Omega_B^mS_{N+1}
   \simeq
   \Omega_B^{m-1}G
   \oplus\Omega_T^{m-1}U
   \oplus\Omega_T^{m-1}V.
\]
The functor \(F\) is additive and restricts to the identity on modules supported on the vertices of \(T\), by \cref{lem:F-exact}.  Hence
\[
   Z_0=\Omega_T^{m-1}U
\]
is a direct summand of \(F(\Omega_B^mS_{N+1})\).

We now record the elementary filtration argument needed below.  For convenience, a degree-zero homomorphism \(f:X\to Y\) is called a \(\mathcal C\)-ghost if
\[
   \underline{\Hom}_H(C,f)=0
   \qquad\text{for every }C\in\mathcal C.
\]
This terminology is adapted from the notion of a ghost in a triangulated category; see \cite{Christensen}.  Equivalently, for every \(g:C\to X\) with \(C\in\mathcal C\), the composite \(fg\) factors through a graded projective module.  Such ghosts form an ideal under composition.  The following is the corresponding statement for filtrations by short exact sequences.  Since this setting differs from the usual triangulated ghost lemma, we include a proof.

\begin{lemma}[Filtration lemma]\label{lem:ghost}
Let \(\mathcal C\) be a full additive subcategory of \(\operatorname{grmod}H\), closed under direct summands.  If
\[
   X\in\operatorname{Filt}_p(\mathcal C),
\]
then every composite of \(p\) \(\mathcal C\)-ghosts with source \(X\) factors through a finitely generated graded projective \(H\)-module.
\end{lemma}
\begin{proof}
We first treat a module \(E\) equipped with a filtration of length at most \(p\), and proceed by induction on \(p\).  For \(p=1\), the assertion follows directly from the definition of a \(\mathcal C\)-ghost.

Assume \(p>1\), and write the first filtration step as a short exact sequence
\[
   0\longrightarrow C\xrightarrow{u}E\xrightarrow{\pi}Y\longrightarrow0,
\]
where \(C\in\mathcal C\) and \(Y\) has a filtration of length at most \(p-1\).  Let
\[
   E\xrightarrow{f_1}X_1\xrightarrow{f_2}\cdots
   \xrightarrow{f_p}X_p
\]
be \(\mathcal C\)-ghosts.  Since \(f_1u\) is stably zero, there are maps
\[
   C\xrightarrow{a}P\xrightarrow{b}X_1
\]
with \(P\) graded projective and \(f_1u=ba\).  The module \(P\) is also graded injective, so \(a\) extends across \(u\) to a map \(\widetilde a:E\to P\).  Hence
\[
   f_1-b\widetilde a
\]
vanishes on \(C\), and therefore factors through \(\pi\), say
\[
   f_1-b\widetilde a=h\pi
\]
for some \(h:Y\to X_1\).  It follows that
\[
   f_p\cdots f_1
   =f_p\cdots f_2h\pi+
     (f_p\cdots f_2b)\widetilde a.
\]
The second summand factors through \(P\).  Moreover, \(f_2h\) is again a \(\mathcal C\)-ghost, because the ghosts form an ideal.  The induction hypothesis applied to \(Y\) shows that the first summand also factors through a graded projective module.  Thus the assertion holds for \(E\).

Finally, suppose that \(X\) is a direct summand of such an \(E\), with maps \(X\xrightarrow{i}E\xrightarrow{q}X\) satisfying \(qi=1_X\).  For any sequence of \(p\) ghosts starting at \(X\), the first map composed with \(q\) is a ghost starting at \(E\).  The preceding paragraph shows that the resulting composite factors through a graded projective; composing with \(i\) gives the original composite.  This proves the lemma.
\end{proof}

For \(0\leq j\leq m+1\), put
\[
   Z_j=I(s-j,h).
\]
For \(0\leq j\leq m\), let
\[
   f_j:Z_j\longrightarrow Z_{j+1}
\]
be the degree-zero map that sends a homogeneous generator \(z_j\) to \(z_{j+1}t\).

\begin{lemma}\label{lem:ghost-maps}
Each \(f_j\) is a \(\mathcal C_m\)-ghost, but
\[
   f_m\cdots f_0:Z_0\longrightarrow Z_{m+1}
\]
does not factor through a graded projective module.
\end{lemma}
\begin{proof}
We check the four classes in \cref{lem:Cm-types}.  A map from a length-one interval to \(Z_j\) has image in the socle, so its composite with multiplication by \(t\) is zero.  Let \(C=I(p,\ell-1)\), and let \(g:C\to Z_j\) be a degree-zero map with
\[
   g(c)=\lambda z_jt^d
\]
on a homogeneous generator \(c\) of degree \(p\).  Degree preservation forces
\[
   d=p-(s-j).
\]
Then
\[
   f_jg(c)=\lambda z_{j+1}t^{d+1}
\]
factors through a graded free module: choose a free generator \(v\) of degree \(\deg c-1\) and factor as
\[
   c\longmapsto vt,
   \qquad
   v\longmapsto\lambda z_{j+1}t^d.
\]
The first map is well defined because \(t^\ell=0\).  For a graded free source the assertion is automatic, since the composite factors through that source.  Finally, an interval generated in degree at least \(r\ell+1\) admits no nonzero degree-zero map to \(Z_j\), whose support ends at degree at most \(r\ell\).  Since every object of \(\mathcal C_m\) is a direct summand of a finite direct sum of the four displayed types, each \(f_j\) is a \(\mathcal C_m\)-ghost.

The composite is multiplication by
\[
   t^{m+1}=t^{h-1},
\]
which is nonzero on a length-\(h\) interval.  It does not factor through a graded free module.  Indeed, since \(\ell=2h\), the image of every map \(H/(t^h)\to H\) lies in
\[
   \operatorname{ann}_H(t^h)=t^hH,
\]
and every subsequent map from \(H\) to \(H/(t^h)\) kills this ideal.  The same argument applies componentwise to every finite direct sum of graded free modules.  Hence the composite does not factor through a graded projective module.
\end{proof}

\begin{proposition}\label{prop:Br-lower}
One has \(\ddell_BS_{N+1}>2r\).
\end{proposition}
\begin{proof}
If \(\ddell_BS_{N+1}\leq m=2r\), then \cref{lem:ddell-to-extension} places \(F(\Omega_B^mS_{N+1})\), and hence its direct summand \(Z_0\), in \(\operatorname{Filt}_{m+1}(\mathcal C_m)\).  By \cref{lem:ghost}, every composite of \(m+1\) \(\mathcal C_m\)-ghosts from \(Z_0\) would factor through a graded projective module.  This contradicts \cref{lem:ghost-maps}.
\end{proof}

\begin{proof}[Proof of \cref{thm:unbounded-gap}]
By \cref{prop:Br-Findim}, \(\Findim(B_r^{\op})=1\).

\Cref{prop:Br-lower}, together with \eqref{eq:Br-upper}, gives
\[
   2r<\ddell_{B_r}S_{N+1}
   \leq\ddell B_r\leq2r+1.
\]
Since \(\ddell B_r\) is integer-valued, \(\ddell B_r=2r+1\).  Therefore
\[
   \ddell B_r-\Findim(B_r^{\op})=2r,
\]
which is unbounded as \(r\to\infty\).
\end{proof}

\begin{corollary}\label{cor:Question51}
The equality
\[
   \ddell A=\Findim(A^{\op})
\]
fails in general, and the difference between its two sides is unbounded, even among finite-dimensional monomial algebras.  Equivalently, the quantities
\[
   \ddell(\Gamma^{\op})-\Findim\Gamma
\]
are unbounded among finite-dimensional monomial algebras \(\Gamma\).
\end{corollary}
\begin{proof}
The first assertion follows by taking \(A=B_r\) in \cref{thm:unbounded-gap} and letting \(r\) vary.  For the equivalent formulation, take \(\Gamma=B_r^{\op}\); the opposite of a finite-dimensional monomial algebra is again a finite-dimensional monomial algebra.
\end{proof}

\paragraph*{A related tensor-product question.}
Lam \cite[Question~4.3]{LamTensor} asks whether
\[
   \ddell(A\otimes_K A)-2\ddell A
\]
can be unbounded on a family of finite-dimensional algebras.  The family in \cref{thm:unbounded-gap} concerns instead the difference \(\ddell A-\Findim(A^{\op})\).  The present calculations do not determine the derived delooping levels of the tensor squares \(B_r\otimes_K B_r\).

\bigskip
\paragraph*{Acknowledgements.}
The new examples constructed in this paper were developed with the assistance of ChatGPT.

\end{document}